\documentclass[a4paper, 12pt]{amsart}

\usepackage{ucs}
\usepackage[utf8x]{inputenc}
\usepackage[english]{babel}
\usepackage{amsfonts,amsmath, amssymb}
\usepackage{fancyhdr}
\usepackage[fixlanguage]{babelbib}
\usepackage{url}
\usepackage{tikz-cd}
\usepackage{hyperref}
\usepackage{color}
\usepackage{mathtools, tensor, faktor}
\usepackage{cite}
\usepackage{esint}
\usepackage{cancel}
\usepackage{tikzsymbols}
\usepackage{mathabx}

\makeatletter
\DeclareRobustCommand*{\mfaktor}[3][]
{
   { \mathpalette{\mfaktor@impl@}{{#1}{#2}{#3}} }
}
\newcommand*{\mfaktor@impl@}[2]{\mfaktor@impl#1#2}
\newcommand*{\mfaktor@impl}[4]{
   \settoheight{\faktor@zaehlerhoehe}{\ensuremath{#1#2{#3}}}%
   \settoheight{\faktor@nennerhoehe}{\ensuremath{#1#2{#4}}}%
      \raisebox{-0.5\faktor@zaehlerhoehe}{\ensuremath{#1#2{#3}}}%
      \mkern-4mu\diagdown\mkern-5mu
      \raisebox{0.5\faktor@nennerhoehe}{\ensuremath{#1#2{#4}}}%
}
\makeatother

\newtheorem{thm}{Theorem}[section]
\newtheorem{lem}[thm]{Lemma}
\newtheorem{cor}[thm]{Corollary}

\theoremstyle{definition}
\newtheorem{defi}[thm]{Definition}
\newtheorem{ex}[thm]{Example}

\theoremstyle{remark}
\newtheorem{rem}[thm]{Remark}
\newtheorem{notation}[thm]{Notation}

\newcommand{\eps}{\varepsilon}
\newcommand{\mfdspace}{\mathcal{M}(n,d)}
\newcommand{\mfdspacenk}{\mathcal{M}(n+k,d)}
\newcommand{\dGH}{d_{\text{GH}}}

\newcommand{\bbN}{\mathbb{N}}

\newcommand{\bbR}{\mathbb{R}}
\newcommand{\bbC}{\mathbb{C}}
\newcommand{\bbZ}{\mathbb{Z}}
\newcommand{\bbT}{\mathbb{T}}

\newcommand{\rO}{\mathrm{O}}

\newcommand{\SpinB}{\tensor[^{\diamond}]{\boldsymbol{\Sigma}}{} B}

\newcommand{\SpinBi}{\tensor[^{\diamond}]{\boldsymbol{\Sigma}}{_i} B}

\DeclareMathOperator{\aff}{aff}
\DeclareMathOperator{\Aff}{Aff}
\DeclareMathOperator{\arsinh}{arsinh}
\DeclareMathOperator{\Aut}{Aut}
\DeclareMathOperator{\diam}{diam}

\DeclareMathOperator{\dom}{dom}
\DeclareMathOperator{\dvol}{dvol}
\DeclareMathOperator{\End}{End}
\DeclareMathOperator{\GL}{GL}
\DeclareMathOperator{\grad}{grad}
\DeclareMathOperator{\Hol}{Hol}

\DeclareMathOperator{\id}{Id}
\DeclareMathOperator{\Pin}{Pin}
\DeclareMathOperator{\pin}{pin}
\DeclareMathOperator{\rank}{rank}
\DeclareMathOperator{\scal}{scal}
\DeclareMathOperator{\SO}{SO}
\DeclareMathOperator{\Span}{Span}
\DeclareMathOperator{\Spin}{Spin}
\DeclareMathOperator{\spin}{spin}

\DeclareMathOperator{\tr}{tr}
\DeclareMathOperator{\vol}{vol}


\begin{document}
\title[Dirac operator and collapse]{The Dirac operator under collapse to a smooth limit space}
\author{Saskia Roos}
\begin{abstract}
Let $(M_i, g_i)_{i \in \bbN}$ be a sequence of spin manifolds with uniform bounded curvature and diameter that converges to a lower dimensional Riemannian manifold $(B,h)$ in the Gromov-Hausdorff topology. Lott showed that the spectrum converges to the spectrum of a certain first order elliptic differential operator $\mathcal{D}^B$ on $B$. In this article we give an explicit description of $\mathcal{D}^B$ and characterize the special case where $\mathcal{D}^B$ is the Dirac operator on $B$. 
\end{abstract}
\maketitle

\section{Introduction}

Let $\mathcal{M}(m,d)$ be the space of closed $m$-dimensional Riemannian manifolds $(M,g)$ with $\vert \sec \vert \leq 1$ and $\diam \leq d$. Gromov proved that any sequence in $\mathcal{M}(m,d)$ contains a subsequence that converges with respect to the Gromov-Hausdorff distance to a compact metric space \cite{GromovPreComp}. There it can happen that the dimension of the limit space is strictly less than $m$. This phenomenon is called \textit{collapsing}. One of the first nontrivial examples of collapse with bounded curvature and diameter was pointed out by Marcel Berger in 1962. He considered the Hopf fibration $S^1 \rightarrow S^3 \rightarrow S^2$. Scaling the fibers by an $\eps >0$ one obtains a collapsing sequence with bounded sectional curvature that converges to a round two-sphere of radius $\frac{1}{2}$ as $\eps \rightarrow 0$. 

The structure of collapse in $\mathcal{M}(m,d)$ was studied by Cheeger, Fukaya and Gromov,\cite{CheegerGromovCollapse1, CheegerGromovCollapse2, FukayaCollapse1, FukayaBoundary, FukayaCollapse2, CheegerFukayaGromov}. Roughly speaking the authors showed that a sequence $(M_i, g_i)_{i \in \bbN}$ that collapses to a metric space $B$ can be approximated by a sequence of singular fibrations $f_i: (M_i, \tilde{g}_i) \rightarrow B$ such that the fibers are infranilmanifolds and $f_i$ is almost everywhere Riemannian submersions.

An interesting question is now, how do the spectra of geometric operators behave under collapse? One would like to know if the spectrum converges and how the limit spectrum is related to the spectrum of the corresponding geometric operator on the limit space.

For the Laplacian on functions, Fukaya showed that if a sequence $(M_i, g_i)_{i \in \bbN}$ in $\mathcal{M}(m,d)$ converges with respect to the measured Gromov-Hausdorff topology to a compact metric space $B$, then the spectrum of the Laplacian on $(M_i,g_i)_{i \in \bbN}$ converges to the spectrum of the Laplacian  on $B$ with respect to the limit measure, \cite{FukayaLaplace}. In general this limit measure is different to the usual volume measure and the spectrum of the associated Laplacian differs from the spectrum of the Laplacian with respect to the standard measure. In \cite{LottLaplaceSingular, LottLaplaceSmooth}, Lott generalized this behavior to the Laplacian on $p$-forms.  Combining the results of \cite{LottLaplaceSingular, LottLaplaceSmooth} with Bochner-type formulas for Dirac-type operators, Lott proved similar results for the spectrum of Dirac-type operators on $G$-Clifford bundles, where $G \in \lbrace \SO(m), \Spin(m) \rbrace$. The results of \cite{LottDirac} can be briefly summarized as follows. Let $(M_i,g_i)_{i \in \bbN}$ be a sequence of manifolds with a $G$-structure in $\mathcal{M}(m,d)$ that converges to a lower dimensional space $B$. Then the spectrum of the Dirac operator restricted to the subspace of spinors that are ``invariant'' along the fibers of the fibrations $M_i \rightarrow B$ converges to the spectrum of an elliptic first order differential operator $\mathcal{D}^B = \sqrt{\Delta + \mathcal{V}}$ acting on a Clifford bundle over $B$. Here $\Delta$ is the Laplacian with respect to the limit measure and $\mathcal{V}$ is a symmetric potential. The remaining part of the spectrum goes to $\pm \infty$ in the limit $i \rightarrow \infty$.

The aim of this paper is to characterize the limit operator $\mathcal{D}^B$ in more detail. We restrict our attention to collapsing sequences of spin manifolds with smooth limit space. This paper is a continuation of \cite{RoosDirac} where we discussed the special case of collapsing sequences of spin manifolds losing one dimension in the limit. In this special case any collapsing sequence can be approximated by a sequence of $S^1$-bundles with a local isometric $S^1$-action. In the general case any collapsing sequence $(M_i, g_i)_{i \in \bbN}$ with a smooth limit space $(B,h)$ can be realized as a collapsing sequence $(f_i: M_i \rightarrow B )_{i \in \bbN}$ of fiber bundles with infranil fibers and affine structure group \cite{FukayaCollapse1, FukayaCollapse2}. This property allows us to prove

\begin{thm}\label{ThmIntroduction}
Let $(M_i, g_i)_{i \in \bbN}$ be a sequence of spin manifolds in $\mfdspacenk$ converging to a closed $n$-dimensional Riemannian manifold $(B,h)$. Then there is a subsequence $(M_i, g_i)_{i \in \bbN}$ such that for all $i \in \bbN$ the space of $L^2$-spinors on $M_i$ can be decomposed into
\begin{align*}
L^2(\Sigma M_i) = \mathcal{S}_i \oplus \mathcal{S}_i^{\perp} 
\end{align*}
such that all eigenvalues of the Dirac operator on $M_i$ restricted to $\mathcal{S}_i^{\perp} $ go to $\pm \infty$ as $i \rightarrow \infty$ and the eigenvalues of the Dirac operator on $M_i$ restricted to $\mathcal{S}_i$ converge to the spectrum of the self-adjoint elliptic first-order differential operator
\begin{align*}
\mathcal{D}^B = \check{D}^{\mathcal{T}} + H
\end{align*}
acting on a twisted Clifford bundle $\mathfrak{P}$ over $B$. Here, $\check{D}^{\mathcal{T}}$ is a Dirac operator on $\mathfrak{P}$ and $H$ a $C^{0,\alpha}$-symmetric potential for $\alpha \in [0,1)$.
\end{thm}

In fact, we will give a complete description of the twisted Clifford bundle $\mathfrak{P}$ and of the potential $H$, see Theorem \ref{MainTheorem}. We show that the following three geometric objects of the fiber bundles $f_i: M_i \rightarrow B$ contribute to $\mathcal{D}^B$: The holonomy of the vertical distribution, the integrability of the horizontal distributions and the intrinsic curvature of the fibers. These three different conditions are independent from each other as can be seen in the Examples \ref{ExampleWithoutHolonomy}, \ref{ExampleWithHolonomy}, \ref{ExampleHeisenberg}, \ref{ExampleNonVanishingA}. As a corollary we  identify collapsing sequences where $\mathcal{D}^B$ is the Dirac operator on $B$. 

The paper is structured as follows.  In Section \ref{Section GeometryAffineFiberBundles} we explain how the results of \cite{CheegerFukayaGromov} imply that any collapsing sequence in $\mfdspace$ can be approximated by a sequence of fiber bundles whose fibers are infranilmanifolds $Z$ and the structure group lies in group of the affine diffeomorphisms, $\Aff(Z)$. Thus, we study  the geometry of these fiber bundles in great detail.  In Section \ref{SectionAffineFiberBundles} we first show that the spin structure on the total space of a fibration $M \rightarrow B$ does not imply that $B$ has a spin structure or is even orientable. Nevertheless we show how the space of ``invariant'' spinors on $M$ can be realized as spinors of a twisted spinor bundle over $B$.  In Section \ref{SectionMainTheorem} we combine all the results of the previous sections to prove the main result about the characterization of the convergent part of the Dirac spectra. As a conclusion we characterize the special case, where the convergent part of the Dirac spectra converges to  the spectrum of the Dirac operator on the limit space. 

At this point we want to remark that this paper is an excerpt of the author's thesis \cite{RoosDiss}.

\subsection*{Acknowledgments}
I would like to thank my supervisors Werner Ballmann and Bernd Ammann for many enlightning discussions. My great thanks also go to Andrei Moroianu for his invitation to Orsay and his good explanations about spin structures on fiber bundles. I also thank Alexander Strohmaier for showing me how eigenvalues can be computed numerically. Furthermore, I thank the Max-Planck Institute for providing excellent working conditions. This research was financed by the Hausdorff Research Institute for Mathematics.

\section{Geometry of Riemannian affine fiber bundles} \label{Section GeometryAffineFiberBundles}

Cheeger, Fukaya, and Gromov studied the structure of collapsing sequences with bounded sectional curvature in great detail and full generality, see \cite{CheegerFukayaGromov} and references therein.  Restricted to our special case of sequences $(M_i, g_i)_{i \in \bbN}$ in $\mfdspacenk$ converging to an $n$-dimensional Riemannian manifold $(B,h)$ we conclude from \cite{CheegerFukayaGromov} the following

\begin{thm}\label{ThmInvariantMetrics}
Let $(M_i,g_i)_{i \in \bbN}$ be a collapsing sequence in $\mathcal{M}(m,d)$. Suppose that this sequence converges to a Riemannian manifold $(B,h)$ with respect to the Gromov-Hausdorff topology. Then, for any $i$ sufficiently large, there is a metric $\tilde{g}_i$ on $M_i$ and a metric $\tilde{h}_i$ on $B$ such that
\begin{align*}
\lim_{i \rightarrow \infty} \Vert g_i - \tilde{g}_i \Vert_{C^1} &= 0, \\
 \lim_{i \rightarrow \infty} \Vert h - \tilde{h}_i \Vert_{C^1}&= 0,
\end{align*}
and $f_i: (M_i, \tilde{g}_i) \rightarrow (B, \tilde{h}_i)$ is a \textit{ Riemannian affine fiber bundle}, i.e.\
\begin{itemize}
	\item $f_i$ is a Riemannian submersion,
	\item for each $p$ the fiber $Z_p \coloneqq f_i^{-1}(p)$ is an infranilmanifold with an induced affine parallel metric $\hat{g}_p$,
	\item the structure group lies in $\Aff(Z)$.
\end{itemize}
Moreover, the second fundamental forms of the fibers of the Riemannian subermsions $f_i$ are uniformly bounded in norm by a positive constant $C(m)$ and the sectional curvature of $(M_i, \tilde{g}_i)_{i \in \bbN}$ are uniformly bounded by a positive constant $K(m)$.
\end{thm}

We shortly recall that an \textit{infranilmanifold} $Z$ is a compact quotient  $\mfaktor{\Gamma}{N}$ of a connected simply-connected nilpotent Liegroup $N$ by a discrete subgroup $\Gamma$ of $\Aff(N) = N_L \rtimes \Aut(N)$, where $N_L$ is the group of left translations acting on $N$ and $\Aut(N)$ is the automorphism group. It follows that the canonical affine connection $\nabla^{\aff}$ on $N$, for which all left invariant vector fields are parallel, descends to a connection $\nabla^{\aff}$ on $Z$. The group $\Aff(Z)$ consists of those diffeomorphisms of $Z$ that preserve $\nabla^{\aff}$ and  we call a tensor field on $Z$ \textit{affine parallel} if it is parallel with respect to $\nabla^{\aff}$. For a thorough introduction to infranilmanifolds we refer to \cite{Dekimpe}, see also \cite[Section 3]{CheegerFukayaGromov}, \cite[Section 3]{LottLaplaceSmooth}.

\begin{proof}[Proof of Theorem \ref{ThmInvariantMetrics}]
By \cite[Theorem 10.1]{FukayaBoundary} the metric $h$ on $B$ is $C^{1, \alpha}$. Furthermore, for $i$ sufficiently large, there is a fibration $f_i: (M_i,g_i) \rightarrow (B,h)$ such that the fiber is an infranilmanifold $Z_i$. Next, we apply the smoothing result by Abresch (see for instance \cite[Theorem 1.12]{CheegerFukayaGromov}): For any positive $\delta$ there is a smooth Riemannian metric $\bar{g}$ such that
\begin{itemize}
\item $\Vert g - \bar{g} \Vert_{C^1} < \delta$,
\item $\Vert \bar{\nabla}^i \bar{R} \Vert_{C^0} < A_i(m,\delta)$ for some positive constants $A_i(m,\delta)$.
\end{itemize}
By \cite[Proposition 2.5]{RongSectionalCurv} there is a positive constant $c(m)$ such that
\begin{align*}
\vert \sec(g) - \sec(\bar{g}) \vert \leq c(m)\delta
\end{align*}
for any sufficiently small positive $\delta$. Now we can apply \cite[Proposition 3.6 and 4.9]{CheegerFukayaGromov} to the fibration $f_i:(M_i, \bar{g}_i) \rightarrow (B,h)$. We obtain another metric $\tilde{g}_i$ with $\Vert \bar{g} - \tilde{g} \Vert_{C^1} \leq \tilde{C}(m) \dGH(M,B)$,  for a positive constant $\tilde{C}(m)$, such that $f_i:(M_i, \tilde{g}_i) \rightarrow (B,\tilde{h}_i)$ is a Riemannian affine fiber bundle for an induced metric $\tilde{h}_i$ and such that the second fundamental forms of fiber is bounded by a constant $C(m)$. Letting $\delta$ go to $0$ as $i$ goes to infinity concludes the proof. 
\end{proof} 

In particular, it suffices to study collapsing sequences of Riemannian affine fiber bundles because Dirac eigenvalues are continuous under a $C^1$-change of the metric \cite[Main Theorem 2]{Nowaczyk}.

\subsection{Useful operators on Riemannian affine fiber bundles}\label{SectionOperatorsRiemannianAffine}

Let $f: (M,g) \rightarrow (B,h)$ be a Riemannian affine fiber bundle with infranil fiber $Z$. From now on, we set $\dim(B) = n$ and $\dim(Z) = k$ for positive integers $n$, $k$. Since $f$ is a Riemannian submersion $TM = \mathcal{H} \oplus \mathcal{V}$, where $\mathcal{H}$ is the horizontal distribution isometric to $f^{\ast}TB$ and $\mathcal{V} = \ker(df )$ is the vertical distribution. The relations between the curvatures of $(M,g)$,$(B,h)$ and the fibers $(Z,\hat{g})$ are given by O'Neill's formulas, see for instance \cite[Theorem 9.28]{Besse}. These formulas involve the two tensors $T$ and $A$ defined via
\begin{align*}
T(X,Y) &\coloneqq \left(\nabla_{X^V} Y^V \right)^H + \left( \nabla_{X^V} Y^H \right)^V, \\
A(X,Y) &\coloneqq \left(\nabla_{X^H} Y^V \right)^H + \left( \nabla_{X^H} Y^H \right)^V,
\end{align*}
for all vector fields $X$, $Y \in \Gamma(TM)$. Here $X^V, X^H$ denote the vertical, resp.\ horizontal part. Roughly speaking, the $T$-tensor is related to the second fundamental form of the fibers and the $A$-tensor vanishes if and only if the horizontal distribution $\mathcal{H}$ is integrable. In what follows many calculations are carried out in a local orthonormal frame chosen as follows:

\begin{defi}\label{SplitFrame}
Let $f:(M,g) \rightarrow (B,h)$ be a Riemannian affine fiber bundle. For all $x \in M$ a local orthonormal frame $(\xi_1, \ldots, \xi_n, \zeta_1, \ldots, \zeta_k)$ is a \textit{split orthonormal frame} if $(\xi_1, \ldots, \xi_n)$ is the horizontal lift of a local orthonormal frame $(\check{\xi}_1, \ldots, \check{\xi}_n)$ around $p = f(x) \in B$ and $(\zeta_1, \ldots, \zeta_k)$ are locally defined affine parallel vector fields tangent to the fibers. 
\end{defi}

Henceforth we label the vertical components  $a, b, c, \ldots$, and the horizontal components $\alpha, \beta, \gamma, \ldots$. The Christoffel symbols with respect to a split orthonormal frame $(\xi_1, \ldots, \xi_n, \zeta_1, \ldots, \zeta_k)$ can be calculated with the Koszul formula:
\begin{equation}\label{ChristoffelSymbols}
	\begin{gathered}
		\Gamma_{ab}^c = \hat{\Gamma}_{ab}^c, \\
		\Gamma_{ab}^{\alpha} = - \Gamma_{a \alpha}^b = g(T(\zeta_a,\zeta_b), \xi_{\alpha} ), \\
		\Gamma_{\alpha a}^b = g( [\xi_{\alpha}, \zeta_a], \zeta_b ) + g(T(\zeta_a, \xi_{\alpha}), \zeta_b ),\\
		\Gamma_{\alpha \beta}^a = - \Gamma_{\alpha a}^{\beta} = - \Gamma_{a \alpha}^{\beta} = g(A(\xi_{\alpha}, \xi_{\beta}), \zeta_a ), \\
		\Gamma_{\alpha \beta}^{\gamma} = \check{\Gamma}_{\alpha \beta}^{\gamma}.
	\end{gathered}
\end{equation}
Here $\hat{\Gamma}_{ab}^c$ are the Christoffel symbols of the fiber $(Z, \hat{g})$ and $\check{\Gamma}_{\alpha \beta}^{\gamma}$ are the Christoffel symbols of $(B,h)$. 

\begin{rem}
The above equations for the Christoffel symbols hold for any Riemannian submersion $f: M \rightarrow B$ if we choose the orthogonal frame similar to Definition \ref{SplitFrame}. The only difference is that we only assume that $(\zeta_1, \ldots, \zeta_k)$ is an orthonormal family of vector fields tangent to the fibers.
\end{rem}

For later use we consider the following two operators characterized by their action on vector fields $X, Y$.
\begin{equation*}
\begin{aligned}
\nabla^Z_{X}Y \coloneqq \left( \nabla_{X^V} Y^V \right)^V, \\
\nabla^{\mathcal{V}}_X Y \coloneqq \left( \nabla_{X^H} Y^V \right)^V. 
\end{aligned}
\end{equation*}
We observe that for each $p \in B$, $\nabla^Z$ restricted to a fiber $Z_p$ is the Levi-Civita connection with respect to the induced metric $\hat{g}_p$.  Since $\hat{g}_p$ is by assumption affine parallel, it follows that $\nabla^Z$ preserves the space of affine parallel vector fields. The difference $\mathcal{Z} \coloneqq \nabla^Z - \nabla^{\aff}$ is a one-form with values in $\End(TZ)$, where we view $TZ$ as a vector bundle over $M$. Further,  $\mathcal{Z} = 0$ if and only if the induced metric $\hat{g}_p$ is flat for all $p \in B$. Next, we interpret $\nabla^{\mathcal{V}}$ as a connetion of the vertical distribution $\mathcal{V}$. Observe that $\nabla^{\mathcal{V}}$ preserves the space of affine parallel vector fields.

As any affine vector field on an infranilmanifold  $Z= \mfaktor{\Gamma}{N}$ lifts to a left invariant vector field on the universal cover $N$, it follows that the space of affine parallel vector fields on $Z$  is finite dimensional. Thus, there is a vector bundle $P$ over $B$ such that, for any $p \in B$, the fiber $P_p$ is given by all affine parallel vector fields of the infranilmanifold $Z_p$. By the discussion above we conclude that $\mathcal{Z}$ descends to a well-defined operator on $P$ and $\nabla^{\mathcal{V}}$ induces a connection of $P$. In addition, there is an $\mathcal{A} \in \Omega^2(B, P)$ characterized by 
\begin{align*}
\mathcal{A}(X,Y) = A(\tilde{X},\tilde{Y}),
\end{align*}
for any vector fields $X,Y \in \Gamma(TB)$. Here $\tilde{X}$ denotes the horizontal lift. 

\subsection{Uniform estimates on Riemannian affine fiber bundles}

In the proof of Theorem \ref{ThmIntroduction} it will be shown that exactly the three operators $\nabla^{\mathcal{V}}$, $\mathcal{Z}$ and $\mathcal{A}$, introduced above, contribute additionally to the limit of Dirac operators on a collapsing sequence of spin manifolds in $\mfdspacenk$. To ensure the continuity of the corresponding spectra, we will choose subsequence such that these three operators converge  on $B$ in the $C^0$-topology. Our strategy is to prove uniform $C^1(B)$-bounds. Then we use the compact embedding $C^1\hookrightarrow C^{0,\alpha}$, for $\alpha \in [0,1)$ to choose strong $C^{0,\alpha}$-convergent subsequences. 

Let $f: (M,g) \rightarrow (B,h)$ be a Riemannian affine fiber bundle. The $C^1(B)$-bounds on $\nabla^{\mathcal{V}}$, $\mathcal{Z}$,  and $\mathcal{A}$ will depend on the following three bounds
\begin{align*}
\Vert A \Vert_{\infty} \leq C_A, \quad 
\Vert T \Vert_{\infty} \leq C_T, \quad 
\Vert R^M \Vert_{\infty} \leq C_R.
\end{align*}
The next lemma shows that such constants exist uniformly for any sequence $(M_i,g_i)_{i \in \bbN}$ in $\mfdspacenk$ converging to an $n$-dimensional Riemannian manifold $(B,h)$. In particular, we can restrict without loss of generality to the case of collapsing Riemannian affine fiber bundles as all of the bounds derived in the next sections are also valid in the case of collapsing sequences in $\mfdspacenk$ with smooth $n$-dimensional limit space.

\begin{lem}\label{SubmersionBounds}
Let $(M_i, g_i)_{i \in \bbN}$ be a sequence in $\mfdspacenk$ converging to an $n$-dimensional Riemannian manifold $(B,h)$. Then there is an index $I$ such that for all $i \geq I$ there are metrics $\tilde{g}_i$ on $M_i$ and $\tilde{h}_i$ on $B$ such that $f_i:(M_i, \tilde{g}_i) \rightarrow (B,\tilde{h}_i)$ is a Riemannian affine fiber bundle and
\begin{equation}\label{MetricLimits}
\begin{aligned}
\lim_{i \rightarrow \infty} \Vert \tilde{g}_i - g_i  \Vert_{C^1} &= 0,\\
\lim_{i \rightarrow \infty} \Vert \tilde{h}_i - h \Vert_{C^1} &= 0.
\end{aligned} 
\end{equation}
In particular, there is a positive constant $C_R(n+k)$, such that $\vert \sec^{\tilde{g}_i} \vert \leq C_R$ for all $i \geq I$. Moreover, there are positive constants $C_A(n,k,B)$, $C_T(n+k)$ such that the fundamental tensors $A_i$ and $T_i$ of the Riemannian submersion $f_i:(M_i, \tilde{g}_i) \rightarrow (B,\tilde{h}_i)$ are uniformly bounded in norm, i.e.\ for all $i \geq I$,
\begin{align*}
\Vert A_i \Vert_{\infty} &\leq C_A,\\
\Vert T_i \Vert_{\infty} &\leq C_T.
\end{align*}
\end{lem}

\begin{proof}
Applying Theorem \ref{ThmInvariantMetrics}, there is an index $I$ such that for all $i \geq I$ there are metrics $\tilde{g}_i$ on $M_i$, and $\tilde{h}_i$ on $B$ such that $f_i: (M_i, \tilde{g}_i) \rightarrow (B,\tilde{h}_i)$ is a Riemannian affine fiber bundle and the metrics $(\tilde{g}_i)_{i \geq I}$ and $(\tilde{h}_i)_{i \geq I}$ satisfy \eqref{MetricLimits}. Moreover, there is a positive constant $C_R(n+k)$ such that  $\vert \sec^{\tilde{g}_i} \vert \leq C_R$ for all $i\geq I$ and a uniform bound $C_T(n+k)$ on the $T$-tensor. 

It follows from \cite[Theorem 10.1]{FukayaBoundary} that for all $i \geq I$ there is a Riemannian manifold $(\tilde{B}, \tilde{h}_i^F)$ with an isometric $\rO(n+k)$-action such that $\faktor{\tilde{B}}{\rO(n+k)}$ is isometric to $(B, \tilde{h}_i)$. Furthermore, there is a $\Lambda_1(n+k) > 0$ such that $\vert \sec^{\tilde{h}_i^F} \vert \leq \Lambda_1$ for all $i \geq I$, \cite[Theorem 6.1]{FukayaBoundary}. Since $\faktor{\tilde{B}}{\rO(n+k)}$ is a closed Riemannian manifold, there is a $\Lambda_2(n+k,B) > 0$ such that the sectional curvature of $(B, \tilde{h}_i)$ is uniformly bounded, i.e.\ $\vert \sec^{\tilde{h}_i} \vert \leq \Lambda_2$ for all $i \geq I$. Thus, it follows via O'Neill's formula \cite[Corollary 9.29c]{Besse} that
\begin{align*}
\Vert A_i \Vert^2_{\infty} \leq \frac{n(n-1)}{6} (\vert \sec^{\tilde{g}_i}\vert + \vert \sec^{\tilde{h}_i}\vert ) \leq  \frac{n(n-1)}{6} (C_R(n) + \Lambda_2 ) \eqqcolon C_A^2.
\end{align*}
\end{proof}

\subsubsection{Uniform bounds for $\nabla^{\mathcal{V}}$}

First we deal with $\nabla^{\mathcal{V}}$. Let $(e_1, \ldots, e_k)$ be a local affine parallel frame for the vertical distribution $\mathcal{V}$ such that $[X, e_a] = 0$ for any basic vector field $X$. We write $\langle . , . \rangle$ for a locally defined metric on $\mathcal{V}$ characterized by $\langle e_a, e_b \rangle = \delta_{ab}$. There is a unique positive definite symmetric operator $W$ satisfying
\begin{align}\label{DefWvertmetric}
g(U,V) = \langle W(U), W(V) \rangle,
\end{align}
for all vertical vector fields $U, V$. Recall that the induced metric on the fiber is affine parallel. Hence, $W$ is affine parallel as well. Let $(\xi_1, \ldots, \xi_n, \zeta_1, \ldots, \zeta_k)$ be a split orthonormal frame where $\zeta_a \coloneqq W^{-1}(e_a)$. A short computation shows that
\begin{align*}
g(T(\zeta_a, \xi_{\alpha}), \zeta_b) = \frac{1}{2} \langle(W^{-1}\xi_{\alpha}(W) + \xi_{\alpha}(W)W^{-1} )e_a , e_b\rangle.
\end{align*}
Thus,
\begin{align*}
\Gamma_{\alpha a}^b &= g( [\xi_{\alpha}, \zeta_a], \zeta_b ) + g(T(\zeta_a, \xi_{\alpha}), \zeta_b )\\
&= g( \xi_{\alpha}(W^{-1}) e_a, \zeta_b) + \frac{1}{2} \langle(W^{-1}\xi_{\alpha}(W) + \xi_{\alpha}(W)W^{-1} )e_a , e_b\rangle \\
&= \frac{1}{2}\langle (W^{-1} \xi_{\alpha}(W) - \xi_{\alpha}(W) W^{-1}) e_a, e_b \rangle \eqqcolon \langle \mathcal{W}_{\xi_{\alpha}} e_a, e_b \rangle.
\end{align*}
By abuse of notation, we use the same letter $\mathcal{W}$ for the connection one-form of  $\nabla^{\mathcal{V}}$. To obtain a $C^1(B)$-bound for  $\nabla^{\mathcal{V}}$, it suffices to derive a uniform $C^1$-bound on the connection form $\mathcal{W}$. 

\begin{lem}\label{BoundConnV}
Let $f: (M,g) \rightarrow (B,h)$ be a Riemannian affine fiber bundle such that
\begin{align*}
\Vert A \Vert_{\infty} \leq C_A, \quad 
\Vert T \Vert_{\infty} \leq C_T, \quad 
\Vert R^M \Vert_{\infty} \leq C_R,
\end{align*}
then
\begin{align*}
\Vert \mathcal{W}_{X} \Vert_{\infty} &\leq 2 C_T \Vert X \Vert_{\infty}, \\
\Vert Y (\mathcal{W}_{\Vert X \Vert)} )\Vert_{\infty} &\leq C(C_T, C_A, C_R ) \Vert Y\Vert_{\infty} \Vert X \Vert_{\infty},
\end{align*}
for any basic vector fields $X, Y$.
\end{lem}

\begin{proof}
Let $( \xi_1, \ldots, \xi_{n}, \zeta_1, \ldots, \zeta_k)$ be a split orthormal frame such that $\zeta_a = W^{-1}e_a$, as above.  In particular, $W$ can be viewed as a field of symmetric positive definite matrices.

For the first inequality, we calculate
\begin{align*}
\vert T \vert^2 &= \sum_{\alpha=1}^n \sum_{a,b=1}^k g(T(\zeta_a, \xi_{\alpha}), \zeta_b)^2 \\
&= \frac{1}{4} \sum_{\alpha=1}^n \sum_{a,b=1}^k \langle (\xi_{\alpha}(W) W^{-1} + W^{-1} \xi_{\alpha}(W)) e_a, e_b \rangle^2 \\
&=  \frac{1}{4}\sum_{\alpha=1}^n \vert \xi_{\alpha}(W) W^{-1} + W^{-1} \xi_{\alpha}(W) \vert^2 \\
&= \frac{1}{2} \sum_{\alpha=1}^n \tr \big( (W^{-1} \xi_{\alpha}(W) )^2 \big) + \tr\big(W^{-2} \xi_{\alpha}(W)^2 \big).
\end{align*}
As $W^{-1}$ and $\xi_{\alpha}(W)$ are symmetric, it follows that
\begin{align*}
\tr\big(W^{-2} \xi_{\alpha}(W)^2 \big) = \vert W^{-1} \xi_{\alpha}(W) \vert^2 \geq 0.
\end{align*}
Since $W^{-1}$ is also symmetric and positive definite it has a unique symmetric positive definite square root $C$, i.e.\ $C^2 = W^{-1}$. Replacing $W^{-1}$ by $C^2$ leads to
\begin{align*}
\tr \big( (W^{-1} \xi_{\alpha}(W) )^2 \big) &= \tr (C^2 \xi_{\alpha}(W) C^2 \xi_{\alpha}(W) \big) = \vert C \xi_{\alpha}(W) C \vert^2 \geq 0.
\end{align*}
Thus, 
\begin{align}\label{BoundedWXW}
\frac{1}{2} \sum_{\alpha =1}^n \Vert W^{-1} \xi_{\alpha}(W) \Vert^2 \leq \Vert T \Vert^2 \leq C_T^2.
\end{align}
It follows immediately that
\begin{align*}
\Vert \mathcal{W}_{\xi_{\alpha}} \Vert_{\infty} = \frac{1}{2} \Vert W^{-1} \xi_{\alpha}(W) - \xi_{\alpha}(W) W^{-1}  \Vert \leq 2 C_T.
\end{align*}

For the second inequality we fix a point $x \in M$. Suppose that $(\xi_1, \ldots, \xi_n)$ is the horizontal lift of an orthonormal frame parallel in $f(x)$. We compute
\begin{align*}
\xi_{\beta} ( \mathcal{W}_{\xi_{\alpha}}) &= \frac{1}{2} \Big( W^{-1} \xi_{\beta}\xi_{\alpha}(W) - \xi_{\beta}\xi_{\alpha}(W) W^{-1}\\
& \qquad \quad + \xi_{\alpha}(W)W^{-1} \xi_{\beta}(W) W^{-1}- W^{-1} \xi_{\beta}(W) W^{-1} \xi_{\alpha}(W)\Big).
\end{align*}
By the inequality \eqref{BoundedWXW}, it remains to bound the second derivatives. 

A straight forward calculation shows that
\begin{equation}\label{Tderivative}
\begin{aligned}
g((\nabla_{\xi_{\beta}}T)(\zeta_a,\zeta_b), \xi_{\alpha} ) &= \frac{1}{2} \langle \Big( \xi_{\beta}\xi_{\alpha}(W) W^{-1} + W^{-1}\xi_{\beta}\xi_{\alpha}(W)\\
& \qquad \quad  +\xi_{\beta}(W)W^{-1}\xi_{\alpha}(W) W^{-1}\\
&\qquad \quad  - W^{-1}\xi_{\beta}(W)W^{-1}\xi_{\alpha}(W) \\
& \qquad \quad  - \xi_{\beta}(W)W^{-1}W^{-1}\xi_{\alpha}(W) \\
& \qquad \quad + W^{-1}\xi_{\alpha}(W)\xi_{\beta}(W)W^{-1} \ \Big) e_a, e_b \rangle
\end{aligned}
\end{equation}
for all $1 \leq a,b \leq k$. Since $\xi_{\beta}\xi_{\alpha}(W) = \xi_{\alpha}\xi_{\beta}(W)$, it follows from \eqref{BoundedWXW} that
\begin{align*}
\chi_{\alpha,\beta,i,j} &\coloneqq g((\nabla_{\xi_{\alpha}} T) (\zeta_i, \zeta_j),\xi_{\beta}) - g((\nabla_{\xi_{\beta}} T) (\zeta_i, \zeta_j),\xi_{\alpha}) 
\end{align*}
is bounded by
\begin{align*}
\Vert \chi_{\alpha,\beta,i,j} \Vert \leq 8C_T^2.
\end{align*}
By \cite[9.32]{Besse},
\begin{align*}
\chi_{\alpha,\beta,i,j} = g((\nabla_{\zeta_i}A)(\xi_{\beta}, \xi_{\alpha}), \zeta_j) + g((\nabla_{\zeta_j}A)(\xi_{\beta}, \xi_{\alpha}), \zeta_i).
\end{align*}
Inserting this equality in \cite[9.28d]{Besse} we obtain 
\begin{align*}
g(R^M(\zeta_j, \zeta_i)\xi_{\beta}, \xi_{\alpha}) &= g((\nabla_{\zeta_i}A)(\xi_{\beta}, \xi_{\alpha}), \zeta_j) - g((\nabla_{\zeta_j}A)(\xi_{\beta}, \xi_{\alpha}), \zeta_i)  \\
& \ \ \ +g(A(\xi_{\beta}, \zeta_j), A(\xi_{\alpha}, \zeta_i)) -g(A(\xi_{\beta}, \zeta_i), A(\xi_{\alpha}, \zeta_j))  \\
&\ \ \  -g(T(\zeta_j, \xi_{\beta}),T(\zeta_i, \xi_{\alpha})) + g(T(\zeta_i, \xi_{\beta}),T(\zeta_j, \xi_{\alpha}))\\
&= 2g((\nabla_{\zeta_i}A)(\xi_{\beta}, \xi_{\alpha}), \zeta_j) + \chi_{\alpha,\beta,i,j}\\
& \ \ \ +g(A(\xi_{\beta}, \zeta_j), A(\xi_{\alpha}, \zeta_i)) -g(A(\xi_{\beta}, \zeta_i), A(\xi_{\alpha}, \zeta_j))  \\
&\ \ \  -g(T(\zeta_j, \xi_{\beta}),T(\zeta_i, \xi_{\alpha})) + g(T(\zeta_i, \xi_{\beta}),T(\zeta_j, \xi_{\alpha})).
\end{align*}
Thus,
\begin{align*}
\vert g((\nabla_{\zeta_i}A)(\xi_{\beta}, \xi_{\alpha}), \zeta_j) \vert &\leq \frac{1}{2} C_R + 5C_T^2 + C_A^2 \eqqcolon C_1.
\end{align*}
Next, we consider \cite[9.28c]{Besse},
\begin{align*}
g(R(\zeta_j,\xi_{\beta})\xi_{\alpha}, \zeta_i) &= g((\nabla_{\xi_{\beta}}T)(\zeta_i, \zeta_j), \xi_{\alpha})) -g(T(\zeta_j, \xi_{\beta}),T(\zeta_i, \xi_{\alpha})) \\
& \ \ \ +  g((\nabla_{\zeta_i}A)(\xi_{\beta}, \xi_{\alpha}), \zeta_j)  + g(A(\xi_{\beta}, \zeta_j), A(\xi_{\alpha}, \zeta_i)).
\end{align*}
It follows that
\begin{align*}
\vert g((\nabla_{\xi_{\beta}}T)(\zeta_i, \zeta_j), \xi_{\alpha})) \vert  \leq  C_R + C_T^2 +2C_A^2 + C_1 \eqqcolon C_2.
\end{align*}
Applying this bound and the inequality \eqref{BoundedWXW} to  \eqref{Tderivative} we deduce
\begin{align*}
\Vert W^{-1} \xi_{\beta}\xi_{\alpha}(W) + \xi_{\beta}\xi_{\alpha}(W) W^{-1} \Vert \leq C_2 + 4C_T^2 \coloneqq C_3.
\end{align*}
Using the same strategy as in the proof of the inequality \eqref{BoundedWXW}, it follows that
\begin{align*}
\Vert  W^{-1} \xi_{\beta}\xi_{\alpha}(W) \Vert \leq C_3.
\end{align*}
Collecting everything so far, the claim follows from
\begin{align*}
\Vert \xi_{\beta}(\mathcal{W}_{\xi_\alpha}) \Vert_{\infty} \leq C_3 + 2C_T^2.
\end{align*}
\end{proof}

\begin{rem}\label{RemarkVertHol}
The condition for $\nabla^{\mathcal{V}}$ to be gauge equivalent to the trivial connection is that the holonomy $\Hol(\mathcal{V}, \nabla^{\mathcal{V}})$ is trivial, see for instance \cite[Section 4.3]{Baum}. 
\end{rem}

In the following  example we show that $\nabla^{\mathcal{V}}$ can be the trivial connection although the $T$-tensor is nontrivial.

\begin{ex}\label{ExampleWithoutHolonomy}
Let $M= B \times \bbT^k$ be the trivial $\bbT^k$-bundle over a Riemannian manifold $(B,h)$. In this situation the vertical distribution $\mathcal{V}$ is the trivial vector bundle $B \times \bbR^k$. For any $i \in \bbN$ we endow $M$ with the Riemannian product metric
\begin{align*}
g_i\coloneqq h \oplus \frac{1}{i^2} u^2 \hat{g},
\end{align*}
where $u: B \rightarrow \bbR$ is a fixed smooth function and $\hat{g}$ is the standard flat metric on $\bbT^k$. Then $(M,g_i)_{i \in \bbN}$ is a collapsing sequence with bounded sectional curvature and diameter. Consider the Riemannian submersions $f_i: (M,g_i) \rightarrow (B,h)$. The fibers are embedded flat tori and the horizontal distribution is integrable for all $i \in \bbN$. The $T$-tensor is, for all $i \in \bbN$, given by
\begin{align*}
T_i(U,V) = \frac{\grad(u)}{u} g_i(U,V)
\end{align*}
for any two vertical vectors $U$, $V$. We claim that the induced connection $\nabla^{\mathcal{V}}$ is trivial with respect to an isometric trivialization. To see this claim, we adapt the notation of Lemma \ref{BoundConnV}. Let $e^i_1, \ldots, e^i_k$ be an orthonormal frame for $\hat{g}_i$. Any such choice induces a global vertical frame on $(M,g_i)$. For $W_i$ defined as in \eqref{DefWvertmetric} we obtain
\begin{align*}
W_i &= \frac{u}{i} \id.
\end{align*}
Hence,
\begin{align*}
(\mathcal{W}_i)_{X} &= \frac{1}{2}( W_i^{-1} X(W_i) - X(W_i) W_i^{-1} ) \\
&= \frac{1}{2} \left( \frac{i}{u} \frac{X(u)}{i} - \frac{X(u)}{i}\frac{i}{u} \right) \id \\
&= 0.
\end{align*}
Therefore, $\mathcal{W}_i = 0$ for all $i \in \bbN$, although the $T$-tensor is nontrivial.
\end{ex}

Next, we state an example of a collapsing sequence such that the corresponding connections $\nabla^{\mathcal{V}_i}$ do not converge to a connection that is gauge equivalent to the trivial connection.

\begin{ex}\label{ExampleWithHolonomy}
Consider the two-dimensional torus $\bbT^2$ and choose an element of $\Aut(\bbT^2) \cong \GL(2,\bbZ)$, e.g.\
\begin{align*}
H \coloneqq \begin{pmatrix}
2 & 1 \\
1 & 1
\end{pmatrix}.
\end{align*}
Let $C \coloneqq [0,1] \times \bbT^2$ be the cylinder over $\bbT^2$ and set
\begin{align*}
M \coloneqq \faktor{C}{\sim},
\end{align*}
where we identify $(0,x)$ with $(1, Hx)$ for all $x \in \bbT^2$. This defines a nontrivial $\bbT^2$-bundle $f:M \rightarrow S^1$. There is a Riemannian metric $g = h \oplus \hat{g}$ on $M$ such that $h$ is the standard metric on $S^1$ and $\left( \hat{g}_{\phi} \right)_{\phi \in S^1}$ is a family of flat metrics on $\bbT^2$. Then $(M,g_i)_{i \in \bbN}$ with $g_i \coloneqq h \oplus \frac{1}{i^2} \hat{g}$ defines a collapsing sequence with bounded sectional curvature and diameter such that the induced connection $\nabla^{\mathcal{V}_i}$ is never gauge equivalent to the trivial connection.
\end{ex}

\subsubsection{Uniform bounds for $\mathcal{Z}$}

Next we consider $\mathcal{Z} = \nabla^{Z}- \nabla^{\aff}$. Recall that $\nabla^Z$ restricted to a fiber $Z_p$ is the Levi-Civita connection of $(Z_p, ,\hat{g}_p)$, where $\hat{g}_p$ is the induced affine parallel metric. We also recall that $\mathcal{Z} = 0$ if and only if the induced metric $\hat{g}_p$ is flat for all $p \in B$. 

\begin{lem}\label{BoundFiberCon}
Let $f:(M,g) \rightarrow (B,h)$ be a Riemannian affine fiber bundle such that
\begin{align*}
\Vert A \Vert_{\infty} \leq C_A, \quad 
\Vert T \Vert_{\infty} \leq C_T, \quad 
\Vert R^M \Vert_{\infty} \leq C_R,
\end{align*}
then 
\begin{align*}
\Vert \mathcal{Z} \Vert &\leq C(k,C_T,C_R) ,\\
\Vert X(\mathcal{Z} ) \Vert &\leq C(k, C_t, C_A, C_R) \Vert X \Vert ,
\end{align*}
for all basic vector fields $X$.
\end{lem}

\begin{proof}
Let $(\zeta_1, \ldots, \zeta_k)$ be as in the split orthonormal frame, see Definition \ref{SplitFrame}. We define the associated structural coefficients $\tau^c_{ab}$ via
\begin{align*}
[\zeta_a, \zeta_b ] = \sum_{c=1}^k \tau_{ab}^c \zeta_c.
\end{align*}
These are also the structural constants of the Lie algebra $\mathfrak{n}$ of the nilpotent Lie group $N$ that covers $Z$ with respect to the pullback of  $(\zeta_1, \ldots, \zeta_k)$.

A straightforward calculation using the Koszul formula shows that
\begin{align*}
\Gamma_{ab}^c = \hat{\Gamma}_{ab}^c = \frac{1}{2}\left( \tau_{ab}^b - \tau_{ac}^b - \tau_{bc}^a \right).
\end{align*}

Thus,
\begin{align*}
\Vert \mathcal{Z}\Vert^2 &= \sum_{a,b,c=1}^k (\Gamma_{ab}^c)^2 \\
&= \frac{3}{4} \sum_{a,b,c=1}^k (\tau_{ab}^c)^2+ 2  \sum_{a,b,c=1}^k (\tau_{ac}^b \tau_{bc}^a - \tau_{ab}^c \tau_{ac}^b - \tau_{ab}^c \tau_{bc}^a )\\
&= - 3 \scal(Z),
\end{align*}
because $\sum_{ia,b,c=1}^k \tau_{ab}^c \tau_{ac}^b = 0$ as $\mathfrak{n}$ is nilpotent and $\sum_{a,b,c=1}^k (\tau_{ab}^c)^2 = - 4 \scal(Z)$ by \cite[Lemma 1]{LottLaplaceSmooth}. Now the first inequality follows from O'Neill's formula \cite[9.29c]{Besse},
\begin{align*}
\vert \scal(Z) \vert &\leq \sum_{a,b=1} \vert \sec^Z(\zeta_a, \zeta_b)\vert \\
&= \sum_{a,b=1} \vert \sec^M(\zeta_a, \zeta_b) - \vert T(\zeta_a, \zeta_b )\vert^2 + g(T(\zeta_a, \zeta_a), T(\zeta_b, \zeta_b) )\vert \\
&\leq k^2 (C_R + 2 C_T^2).
\end{align*}

The second inequality is also proven in local coordinates,
\begin{align*}
\vert X(\mathcal{Z}) \vert^2 = \sum_{a,b,c=1}^k \left(X(\Gamma_{ab}^c) \right)^2
\end{align*}
for any basic vector field $X$. We calculate
\begin{align*}
\vert X(\Gamma_{ab}^c) \vert &= \vert g(\nabla_X \nabla_{\zeta_a} \zeta_b, \zeta_c) - g(\nabla_{\zeta_a} \zeta_b, \nabla_X \zeta_c) \vert \\
&= \vert g(R^M(X, \zeta_a)\zeta_b + \nabla_{[X, \zeta_a]}\zeta_b + \nabla_{\zeta_a} \nabla_X \zeta_b, \zeta_c)- g(\nabla_{\zeta_a} \zeta_b, \nabla_X \zeta_c)  \vert \\
&= \vert g(R^M(X, \zeta_a)\zeta_b + \nabla_{[X, \zeta_a]}\zeta_b, \zeta_c) - g(\nabla_X \zeta_b, \nabla_{\zeta_a} \zeta_c )- g(\nabla_{\zeta_a} \zeta_b, \nabla_X \zeta_c)  \vert .
\end{align*}
As  the Lie bracket $[X, \zeta_a]$ is vertical  for any basic vector field $X$ and $1 \leq a \leq k$, we use Lemma  \ref{BoundConnV} to conclude that
\begin{align*}
\vert g(\nabla_{[X, \zeta_a]}\zeta_b, \zeta_c) \vert &\leq \vert [X, \zeta_a] \vert \, \Vert \mathcal{Z} \Vert \\
&= \vert \nabla_X^{\mathcal{V}} \zeta_a + T(\zeta_a, X) \vert \Vert \mathcal{Z} \Vert \\
&\leq 3C_T \Vert \mathcal{Z} \Vert \vert X \vert,
\end{align*}
and
\begin{align*}
\vert g(\nabla_X \zeta_b, \nabla_{\zeta_a} \zeta_c ) \vert &\leq \vert g(\nabla_X^{\mathcal{V}} \zeta_b, \nabla^Z_{\zeta_a} \zeta_c )\vert + \vert g(A(X,\zeta_a), T(\zeta_a, \zeta_c ) ) \vert \\
& \leq (2 C_T \Vert \mathcal{Z} \Vert + C_A C_T ) \vert X \vert,
\end{align*}
hold for any basic vector field $X$. Combining these inequalities we conclude
\begin{align*}
\vert X(\Gamma_{ab}^c) \vert \leq \Big( C_R + 7C_T \Vert \mathcal{Z} \Vert  + 2C_A C_T \Big) \vert X \vert.
\end{align*}
\end{proof}

The following example shows that for Riemannian affine fiber bundles with non-flat fibers the one-form $\mathcal{Z}$ is nontrivial.

\begin{ex}\label{ExampleHeisenberg}
Let $M = \mfaktor{\Gamma}{N}$ be a nilmanifold, where $N$ is the $3$-dimensional Heisenberg group
\begin{align*}
N = \left \lbrace \begin{pmatrix}
1 & x & z \\
0 & 1 & y \\
0 & 0 & 1 
\end{pmatrix} : \ x,y,z \in \bbR \right\rbrace
\end{align*}
and 
\begin{align*}
\Gamma = \left \lbrace \begin{pmatrix}
1 & x & z \\
0 & 1 & y \\
0 & 0 & 1 
\end{pmatrix} : \ x,y,z \in \bbZ \right\rbrace.
\end{align*}
The Lie algebra $\mathfrak{n}$ of $N$ is given by
\begin{align*}
\mathfrak{n} = \left \lbrace \begin{pmatrix}
0 & x & z \\
0 & 0 & y \\
0 & 0 & 0 
\end{pmatrix} : \ x,y,z \in \bbR \right\rbrace.
\end{align*}
We fix the basis
\begin{align*}
X \coloneqq \begin{pmatrix}
0 & 1 & 0 \\
0 & 0 & 0 \\
0 & 0 & 0 
\end{pmatrix}, \
Y &\coloneqq \begin{pmatrix}
0 & 0 & 0 \\
0 & 0 & 1 \\
0 & 0 & 0 
\end{pmatrix}, \
Z \coloneqq \begin{pmatrix}
0 & 0 & 1 \\
0 & 0 & 0 \\
0 & 0 & 0 
\end{pmatrix}
\end{align*}
and let $X^{\ast}, Y^{\ast}, Z^{\ast}$ be the dual basis. For any $i \in \bbN$ we consider the affine parallel metric
\begin{align*}
g_i \coloneqq \frac{1}{i^2} X^{\ast}\cdot X^{\ast} + \frac{1}{i^2} Y^{\ast} \cdot Y^{\ast} + \frac{1}{i^4} Z^{\ast} \cdot Z^{\ast}.
\end{align*}
It is not hard to check that $(Z, g_i)_{i \in \bbN}$ defines a collapsing sequence with bounded curvature that converges to a point as $i$ goes to infinity. 

Since,
\begin{align*}
\left( \mfaktor{\Gamma}{N}, g_1 \right) & \rightarrow \left( \mfaktor{\Gamma}{N}, g_i \right), \\
\begin{pmatrix}
1 & x & z \\
0 & 1 & y \\
0 & 0 & 1 
\end{pmatrix}
&\mapsto
\begin{pmatrix}
1 & ix & i^2 z \\
0 &  1 & i y \\
0 & 0 & 1
\end{pmatrix}
\end{align*}
is an $i^4$-fold isometric covering the metric $g_i$ has the same nontrivial curvature for all $i\in \bbN$. In particular, $\mathcal{Z}_i$ does not vanish in the limit $i \rightarrow \infty$.
\end{ex}

\subsubsection{Uniform bounds for $\mathcal{A}$}

Finally we consider $\mathcal{A} \in \Omega^2(B, \mathcal{P})$. This two-form on a Riemannian affine fiber bundle $f:M \rightarrow B$ is characterized by the property $f^{\ast} \mathcal{A} = A_{\vert \mathcal{H} \times \mathcal{H}}$. In the following example we see that this tensor can be nonzero while $\mathcal{Z} = 0$ and $\nabla^{\mathcal{V}}$ is trivial.

\begin{ex}\label{ExampleNonVanishingA}
Let $f:(M,g) \rightarrow (B,h)$ be an $S^1$-principal bundle such that $f$ is a Riemannian submersion with totally geodesic fibers of length $2 \pi$. Suppose further that the curvature form $\mathcal{A}$ of the $S^1$-principal bundle is nontrivial. Note that for any $i$ the cyclic subgroup $\bbZ_i < S^1$ acts on $M$ by isometries. The sequence $(\faktor{M}{\bbZ_i}, g_i)_{i \in \bbN}$ converges with bounded sectional curvature and diameter to $(B,h)$. Here $g_i$ is the induced quotient metric. In this sequence the fibers are embedded flat manifolds and the holonomy of the vertical bundle is trivial for all~$i$. However, the $A$-tensor of $(\faktor{M}{\bbZ_i}, g_i) \rightarrow (B, h)$ is given by
\begin{align*}
A_i(X,Y) = -\frac{1}{2} \mathcal{A}(X,Y) V,
\end{align*}
where $X,Y$ are horizontal vector fields and $V$ is a vertical vector field of unit length. In particular $A_i$ does not vanish in the limit $i \rightarrow\infty$.
\end{ex}

The following lemma is a generalization of \cite[Lemma 2.7]{RoosDirac}.

\begin{lem}\label{BoundATensor}
Let $f: (M,g) \rightarrow (B,h)$ be a Riemannian affine fiber bundle such that
\begin{align*}
\Vert A \Vert_{\infty} \leq C_A, \quad
\Vert T \Vert_{\infty} \leq C_T,\quad
\Vert R^M \Vert_{\infty} \leq C_R,
\end{align*}
then
\begin{align*}
\Vert \mathcal{A} \Vert_{C^1(B)} \leq C(k,n,C_A,C_T,C_R).
\end{align*}
\end{lem}

\begin{proof}
We use a split orthonormal frame, see Definition \ref{SplitFrame}, and write
\begin{align*}
(f^{\ast} \mathcal{A})(X,Y) = A_{\vert \mathcal{H} \times \mathcal{H}}(X,Y) &= \sum_{a=1}^k g(A(X,Y), \zeta_a) \zeta_a\\
& \eqqcolon \sum_{a=1}^k A^a(X,Y)  \zeta_a \eqqcolon \sum_{a=1}^k (f^{\ast} \mathcal{A}^a)(X,Y) \zeta_a \rlap{ .}
\end{align*}
Then, 
\begin{align*}
\Vert \mathcal{A} \Vert_{C^1(B)} &\leq \sum_{a=1}^k \left( \Vert \mathcal{A}^a \Vert_{\infty} + \Vert \nabla(\mathcal{A}^a) \Vert_{\infty} \right) \\
&\leq k C_A +  \sum_{a=1}^k \Vert \nabla(\mathcal{A}^a) \Vert_{\infty}.
\end{align*}
We calculate the second term pointwise. Let $p \in B$ be arbitrary and $x \in f^{-1}(p)$. Suppose further that $(\xi_1, \ldots, \xi_k)$ is the horizontal lift of a local orthonormal frame $\left( \check{\xi}_1, \ldots, \check{\xi}_n \right)$ that is parallel in $p \in B$. Using \cite[9.32]{Besse} we obtain
\begin{align*}
\vert \nabla (\mathcal{A}^a) \vert^2 &=  \frac{1}{2} \sum_{\alpha, \beta, \gamma = 1}^n \vert (\nabla_{\check{\xi}_{\alpha}} \mathcal{A}^a) (\check{\xi}_{\beta}, \check{\xi}_{\gamma}) \vert^2 \\
&= \frac{1}{2} \sum_{\alpha, \beta, \gamma = 1}^n \vert g \big( (\nabla_{\xi_{\alpha}} A) (\xi_{\beta}, \xi_{\gamma}), \zeta_a \big) \vert^2 \\
&= \frac{1}{2} \sum_{\alpha, \beta, \gamma = 1}^n \Big\vert  g \big( R^{M}( \xi_{\beta}, \xi_{\gamma} )\xi_{\alpha}, \zeta_a \big) -  g \big( A( \xi_{\beta}, \xi_{\gamma}), T(\zeta_a, \xi_{\gamma}) \big) \\
& \qquad \qquad \quad +  g \big( A(\xi_{\gamma}, \xi_{\alpha} ), T(\zeta_a, \xi_{\beta} ) \big) +  g \big( A(\xi_{\alpha}, \xi_{\beta} ) ,T(\zeta_a, \xi_{\gamma}) \big) \Big\vert^2 \\
& \leq \frac{1}{2}  \sum_{\alpha, \beta, \gamma = 1}^n (C_R + 3C_AC_T)^2.
\end{align*}
\end{proof}

\section{Spin-like structures on Riemannian affine fiber bundles}\label{SectionAffineFiberBundles}

In this section we study Riemannian affine fiber bundles $f:(M,g) \rightarrow (B,h)$,  where $(M,g)$ is a spin manifold with a fixed spin structure. First, we discuss whether the spin structure on $M$ induces a structure on $B$ or on the fibers.  As  $f$ is a Riemannian submersion we discuss afterwards how Clifford multiplication of vertical and horizontal vectors acts on the spinors of $M$. In particular, we derive a formula for the Dirac operator on $M$ expressing the influence of the ``vertcial" and ``horizontal" geometry. Next we discuss how the affine connection $\nabla^{\aff}$ on $M$ lifts to a connection on the spinor bundle $\Sigma M$.  This leads to the notion of affine parallel spinors. The main result of the last subsection is that the space of $L^2$-affine parallel spinors is isometric to an explicit given Clifford bundle over the base manifold $B$.

\subsection{Induced structures}\label{SectionInducedStructures}

Let $f:(M,g) \rightarrow (B,h)$ be a Riemannian affine fiber bundle with infranil fiber $Z$. In the remainder of this section we assume that $M$ is a spin manifold with a fixed spin structure. Since for any $p \in B$, the fiber $Z_p = f^{-1}(p)$ is an embedded oriented submanifold with trivial normal bundle there is an induced spin structure on $Z_p$. Moreover, each path in $B$ connecting two points $p,q \in B$ induces an isomorphism between the induced spin structure on $Z_p$ and the induced spin structure on $Z_q$. 

However, there is in general no induced spin structure on $B$, as can be seen in Example \ref{ExampleCP}. There are even examples of Riemannian affine fiber bundles $M \rightarrow B$ where $M$ is spin and $B$ is nonorientable. 

\begin{ex}[{\cite[Example on p.\ 184 ]{LottDirac}}]
Let $M \coloneqq \mathrm{U}(1) \times_{\bbZ_2} S^2$, where $\bbZ_2$ acts on $\mathrm{U}(1)$ via complex conjugation and on $S^2$ via the antipodal map. Then $M$ is spin and $f: M \rightarrow \bbR \mathrm{P}^2$ is a nontrivial $S^1$-bundle.
\end{ex}

Therefore, we also have to deal with pin${}^{\pm}$ structures. Loosely speaking, pin${}^{\pm}$ structures are a generalization of spin structures to a nonorientable setting. In the following, we briefly sketch the basic definitions and properties of $\pin^{\pm}$-structures. For further details, we refer to \cite{KirbyTaylor}, \cite[Appendix A]{GilkeyPin}.

The double cover $\Spin(n) \rightarrow \SO(n)$ can be extended to a double cover of $\mathrm{O}(n)$ in two inequivalent ways, called $\rho^+: \Pin^+(n) \rightarrow \mathrm{O}(n)$ and $\rho^-: \Pin^{-}(n) \rightarrow \mathrm{O}(n)$. As topological spaces $\Pin^+(n)$ and $\Pin^-(n)$ are both isomorphic to $\Spin(n) \sqcup \Spin(n)$ but the group structure of $\Pin^+(n)$ and $\Pin^-(n)$ is different. To see this, we consider the subgroup $\lbrace \id, r \rbrace \subset \mathrm{O}(n)$, where $r$ is a reflection along a hyperplane. Then 
\begin{align*}
\rho^+( \lbrace \id, r \rbrace) &\cong \faktor{\bbZ}{2 \bbZ} \oplus \faktor{\bbZ}{2 \bbZ}, \\
\rho^- ( \lbrace \id, r \rbrace) &\cong \faktor{\bbZ}{4 \bbZ}.
\end{align*}

The notion of pin${}^{\pm}$ structures is an extension of the definition of spin structures to the double covers $\Pin^{\pm}(n) \rightarrow \mathrm{O}(n)$.

\begin{defi}
A pin${}^{\pm}$ structure on an $n$-dimensional Riemannian manifold $(M,g)$ is a Pin${}^{\pm}$-principal bundle $P_{\Pin^{\pm}}B$ that is a double cover of the orthonormal frame bundle $P_{\mathrm{O}}B$, compatible with the double cover $\Pin^{\pm}(n) \rightarrow \mathrm{O}(n)$.
\end{defi}

\begin{ex}
\begin{align*}
\text{The real projective space} \  \bbR\mathrm{P}^n \ \text{is}
\begin{cases}
\pin^+, & \text{if} \ n = 4k,\\
\pin^-, & \text{if} \ n = 4k + 2,\\
\spin, & \text{if} \ n = 4k + 3.
\end{cases}
\end{align*}
\end{ex} 
 
Similar to spin structures, the existence of a pin${}^{\pm}$ structure is a topological property characterized by the vanishing of specific Stiefel-Whitney classes. The proof of the following theorem can be found in \cite[Lemma 1.3]{KirbyTaylor} .
 
\begin{thm}\label{StiefelWhitneyClasses}
A manifold $M$ admits a pin${}^+$ structure if the second Stiefel-Whitney class $w_2(M)$ vanishes and a pin${}^-$ structure if the Stiefel-Whitney classes satisfy the equation $w_2(M) + w_1(M)^2 = 0$. The topological obstruction for a spin structure is the vanishing of $w_2(M)$ and $w_1(M)$.
\end{thm}

The $i$-th Stiefel-Whitney class $w_i$ of a vector bundle $E \oplus F$ satisfies
\begin{align*}
w_i(E \oplus F) = \sum_{k=1}^i w_i(E) \cup w_{i-k}(F),
\end{align*}
where $\cup$ is the cup product. Together with the above characterization of the various structures, we conclude the following lemma, see \cite[Lemma A.1.5]{GilkeyPin}.

\begin{lem}\label{InducedStructures}
Let
\begin{align*}
0 \rightarrow V_1 \rightarrow V_2 \rightarrow V_3 \rightarrow 0
\end{align*}
be a short exact sequence of real vector bundles over a manifold $M$. For any permutation $\lbrace i,j,k \rbrace$ of $\lbrace 1, 2,3 \rbrace$, we have
\begin{enumerate}
	\item if $V_i$ and $V_j$ are spin, there is an induced spin structure on $V_k$,
	\item if $V_i$ is spin and $V_j$ is pin${}^{\pm}$, there is an induced pin${}^{\mp}$ structure on $V_k$,
	\item if $V_i$ is pin${}^{\pm}$ and $V_j$ is pin${}^{\mp}$ and $V_k$ is orientable, then there is an induced spin structure on $V_k$.
\end{enumerate}
\end{lem}

Let $f: (M,g) \rightarrow (B,h)$ be a Riemannian affine fiber bundle and assume that $M$ is spin. We have the following short exact sequence
\begin{align*}
0 \rightarrow f^{\ast} TB \rightarrow TM \rightarrow \mathcal{V} \rightarrow 0.
\end{align*}
Here $\mathcal{V} = \ker(df)$ is the vertical distribution. Applying Lemma \ref{InducedStructures} we conclude that $f^{\ast} TB$ is spin if and only if $\mathcal{V}$ is spin and that $f^{\ast} TB$ is pin${}^{\pm}$ if and only if $\mathcal{V}$ is pin${}^{\mp}$. But a spin or pin${}^{\pm}$ structure on $f^{\ast} TB$ does not induce a corresponding structure on $B$ itself as can be seen in the following example.
\begin{ex}[{\cite[Example on p.\ 184 ]{LottDirac}}]\label{ExampleCP}
Consider $S^5 \rightarrow \bbC P^2$. Then $f^{\ast} w_2(\bbC P^2) \in H^2(S^5, \bbZ_2)$. But $H^2(S^5, \bbZ_2)$ is trivial. Hence, $f^{\ast}w_2(\bbC P^2) = 0$ although $\bbC P^2$ is not spin.
\end{ex}
Nevertheless, if $B$ is spin or pin${}^{\pm}$ then the structure can be pulled back to $f^{\ast} TB$. Therefore, if $B$ is pin${}^{\pm}$ and $\mathcal{V}$ is pin${}^{\mp}$ then there is an induced spin structure on $M$ by Lemma \ref{InducedStructures}.

\subsection{Spin structures and Riemannian submersions}\label{SectionSpinSubmersion}

We start with an elementary discussion of the canonical complex spin representation (see \cite[Chapter I, \S 5]{LawsonMichelsohn} for more details). First, we recall that the group $\Spin(n)$ is contained in the Clifford algebra $\bbC l(n)\coloneqq Cl(\bbC^n)$ of $\bbC^n$.

If $n$ is even, then there is a unique irreducible representation $\chi_n: \bbC l(n) \rightarrow \GL(\Sigma_n)$ and if $n$ is odd, there are two inequivalent irreducible representations $\chi_n^{\pm}: \bbC l (n) \rightarrow \GL(\Sigma_n)$. Here $\Sigma_n$ is a complex vector space of complex dimension $2^{\left[ \frac{n}{2} \right]}$. The \textit{canonical complex spin representation} is defined as
\begin{align*}
\theta_n \coloneqq \begin{cases}
\left. \chi_n \right._{\vert \Spin(n)},  &\text{if $n$ is even,}\\
\left. \chi_n^+ \right._{\vert \Spin(n)}, &\text{if $n$ is odd}.
\end{cases}
\end{align*}
It is important to remark here, that the restrictions $\left. \chi_n^+ \right._{\vert \Spin(n)}$ and $\left. \chi_n^- \right._{\vert \Spin(n)}$ are equivalent to each other, although the non restricted representations $\chi_n^+$ and $\chi_n^-$ are inequivalent.

If $n$ is even the canonical complex spin representation splits $\Sigma_n = \widehat{\Sigma}^+_n \oplus \widehat{\Sigma}^-_n$ such that the restrictions $\left. \theta_n \right._{\vert \widehat{\Sigma}_n^{\pm}}$ are inequivalent irreducible representations of $\Spin(n)$. This splitting corresponds to the $\pm1$ eigenspaces of the \textit{complex volume element}
\begin{align*}
\omega_n^{\bbC} \coloneqq \mathrm{i}^{\left[ \frac{n+1}{2} \right]} \gamma(e_1) \cdots \gamma(e_n).
\end{align*}
Here $(e_1, \ldots, e_n)$ is the standard basis of $\bbR^n$ and $\gamma :\bbR^{n} \rightarrow \GL(\Sigma_{n})$ denotes Clifford multiplication, i.e.\ $\gamma(v)\gamma(w) + \gamma(w)\gamma(v) = -2 \langle v,w \rangle$, where $\langle . , .\rangle$ is the standard scalar product on $\bbR^n$. The map
\begin{align*}
 \Sigma_n = \widehat{\Sigma}^+_n \oplus \widehat{\Sigma}^-_n &\rightarrow \Sigma_n = \widehat{\Sigma}^+_n \oplus \widehat{\Sigma}^-_n, \\
\psi = \psi^+ + \psi^- &\mapsto \bar{\psi}=\psi^+ - \psi^-
\end{align*}
is called \textit{complex conjugation}.

If $n$ is odd the canonical complex spin representation $\theta_n$ is irreducible. The complex volume element $\omega_n^{\bbC}$ acts trivially on $\Sigma_n$. For later use, we want to define $\theta_n^-\coloneqq \left. \chi_n^- \right._{\vert \Spin(n)}$ for which the complex volume element $\omega_n^{\bbC}$ acts as $-1$. With respect to this representation, the Clifford multiplication of $x \in \bbR^n$ acts as $- \gamma(x)$. Recall from the discussion above that the irreducible representations $\theta_n$ and $\theta_n^-$ are equivalent to each other. 

To understand the interplay of the spin structures on a Riemannian submersion $f: M^{n+k} \rightarrow B^n$ we need to discuss how the canonical complex spin representation of $\Spin(n+k)$ behaves under $\Spin(n) \times \Spin(k)$. Since $\dim_{\bbC}(\Sigma_{m}) = 2^{\left[\frac{m}{2} \right]}$ it follows at once that
\begin{align*}
\dim_{\bbC}(\Sigma_{n+k}) = 
\begin{cases}
2 \dim_{\bbC}(\Sigma_n) \dim_{\bbC}(\Sigma_k), & \text{if $n$ and $k$ are odd,}\\
\dim_{\bbC}(\Sigma_n) \dim_{\bbC}(\Sigma_k), & \text{if $n$ or $k$ is even.}
\end{cases}
\end{align*}
Thus, if $n$ or $k$ is even, there is a vector space isomorphism
\begin{align}\label{IsoOneEven}
\Sigma_{n+k} \cong \Sigma_n \otimes \Sigma_k, 
\end{align}
Here $\Sigma_n \otimes \Sigma_k$ is to be understood as the tensor product of two complex vector spaces. We discuss the behavior of Clifford multiplication under this isomorphism later in this section.

Counting dimensions, it follows that such an isomorphism cannot exist if $n$ and $k$ are both odd. In that case, we proceed as follows:

Using the standard basis $(e_1, \ldots, e_{n+k})$ of $\bbR^{n+k}$ we consider the operator
\begin{align*}
\omega_n^{\bbC} \coloneqq \mathrm{i}^{\left[\frac{n+1}{2} \right]} \gamma(e_1) \cdots \gamma(e_n).
\end{align*}
A short computation shows that $(\omega_n^{\bbC})^2 = \id$. Hence, the action of $\omega_n^{\bbC}$ decomposes $\Sigma_{n+k}$ into the two eigenspaces $\Sigma_{n+k}^+$ and $\Sigma_{n+k}^-$ with respect to the eigenvalues $\pm 1$. 

\begin{rem}
If $n$ and $k$ are odd, the splitting $\Sigma_{n+k} = \Sigma_{n+k}^+ \oplus \Sigma_{n+k}^-$ defined above is \textit{different} from the canonical splitting $\Sigma_{n+k} = \widehat{\Sigma}^+_{n+k} \oplus \widehat{\Sigma}^-_{n+k}$ for even dimensions, as $\omega_n^{\bbC}$ and $\omega_{n+k}^{\bbC}$ are not simultaneously diagonalizable.
\end{rem}

Next we observe that the operator
\begin{align*}
\omega_k^{\bbC} \coloneqq \mathrm{i}^{\left[ \frac{k+1}{2} \right]} \gamma(e_{n+1}) \cdots \gamma(e_{n+k})
\end{align*}
anticommutes with $\omega_n^{\bbC}$. Moreover, $(\omega_k^{\bbC})^2 = \id$. Hence, the action of $\omega_k^{\bbC}$ defines an involution
\begin{align*}
\omega_k^{\bbC}: \Sigma_{n+k}^{\pm} \rightarrow \Sigma_{n+k}^{\mp}.
\end{align*}
In the following, we identify $\Sigma^-_{n+k}$ with the image $\omega_k^{\bbC}(\Sigma_{n+k}^+)$. Since
\begin{align*}
\dim_{\bbC}(\Sigma_{n+k}^{\pm}) = \frac{1}{2} \dim_{\bbC}(\Sigma_{n+k}) = \dim_{\bbC}(\Sigma_n) \dim_{\bbC}(\Sigma_k),
\end{align*}
there is an vector space isomorphism 
\begin{align*}
\Sigma_{n+k}^{\pm} \cong \Sigma_n^{\pm} \otimes \Sigma_k.
\end{align*}
Here, the notation $\Sigma_n^{\pm}$ symbolizes that $\omega_n^{\bbC}$ acts as $\pm 1$. Later we will see that, in fact, $\Sigma_n^{\pm}$ corresponds to the two irreducible spinor representations $\theta_n$, $\theta_n^-$. Summarizing the above discussion, we conclude
\begin{equation}\label{IsoOddOdd}
\begin{aligned}
\Sigma_{n+k} &= \Sigma_{n+k}^+ \oplus \Sigma_{n+k}^-, \qquad \qquad \quad \text{if $n$ and $k$ are odd.}\\
&\cong \left(\Sigma_n^+ \otimes \Sigma_k \right) \oplus \left( \Sigma_n^- \otimes \Sigma_k \right)\\
&\cong \left(\Sigma_n^+ \oplus \Sigma_n^- \right) \otimes \Sigma_k.
\end{aligned}
\end{equation}

Next we want to determine how Clifford multiplication with vectors in $\bbR^{n+k}$ ``separates'' into Clifford multiplications on $\Sigma_n$ and $\Sigma_k$. For all natural numbers $n$ and $k$ the Clifford algebra $\bbC l(n + k)$ is canonical isomorphic to the graded tensor product $\bbC l(n) \, \hat{\otimes} \, \bbC l(k)$, endowed with the multiplication
\begin{align*}
(a \, \hat{\otimes}\, \varphi ) \cdot (b \, \hat{\otimes} \, \psi) = (-1)^{\deg( \varphi) \deg(b)} ( a \cdot b ) \, \hat{\otimes} \, (\varphi \cdot \psi),
\end{align*}
see for instance \cite[Proposition 1.12]{SpinorialApproach}. 

If $n$ or $k$ is even, the multiplication of the graded tensor product $\bbC l(n) \, \hat{\otimes} \, \bbC l(k)$ carries over to $\Sigma_n \otimes \Sigma_k$. In the remaining case, $n$ and $k$ odd, we recall the eigenvalue decomposition $\Sigma_{n+k} = \Sigma_{n+k}^+ \oplus \Sigma_{n+k}^-$ together with the involution $\omega_k^{\bbC}: \Sigma_{n+k}^{\pm} \rightarrow \Sigma_{n+k}^{\mp}$. Since $\omega_k^{\bbC}$ anticommutes with $\omega_n^{\bbC}$ it follows that Clifford multiplication with vectors $v \in \Span \lbrace e_1, \ldots, e_n \rbrace$ acts as $\gamma_{n+k}(v)$ on $\Sigma_{n+k}^+$ and as $- \gamma_{n+k}(v)$ on $\Sigma_{n+k}^-$. On the other hand, Clifford multiplication with any vector in $\Span \lbrace e_{n+1}, \ldots, e_{n+k} \rbrace$ interchanges the eigenspaces $\Sigma_{n+k}^{\pm}$ and commutes with $\omega_k^{\bbC}$. Using the isomorphisms \eqref{IsoOneEven} and \eqref{IsoOddOdd} combined with the above discussion, we obtain the following identifications for Clifford multiplication with vectors $(x,v) \in \bbR^n \times \bbR^k$:
\begin{align}\label{CliffordMultCases}
\gamma_{n+k}((x,v))(\psi \otimes \nu ) \cong 
\begin{cases}
(\gamma_n(x) \psi) \otimes \nu + \bar{\psi} \otimes (\gamma_k(v) \nu ), \ \text{ if $n$ is even,} \\
(\gamma_n(x) \psi) \otimes \bar{\nu} + \psi \otimes (\gamma_k(v) \nu), \ \text{ if $k$ is even,} \\
\ \ \\
( \gamma_n(x) \psi^+ \oplus - \gamma_n(x) \psi^-) \otimes \nu + (\psi^- \oplus \psi^+) \otimes (\gamma_k(v) \nu ), \\
\ \text{if $n$ and $k$ are odd.}
\end{cases}
\end{align}
Here, $\gamma_m$ denotes the Clifford multiplication of $\bbR^m$ on $\Sigma_m$ and $\bar{\psi}$ is the complex conjugation. In the case, where $n$ and $k$ are even, both possibilities are isomorphic to each other.

Now we return to the case of a Riemannian submersion $f:(M,g) \rightarrow (B,h)$ where $M$ is an $(n+k)$-dimensional  spin manifold with a fixed spin structure and $\dim(B) = n$.  There is always locally an induced spin structure on $B$ with a locally defined spinor bundle $\Sigma B$ and similarly a locally defined induced spin structure on the vertical distribution $\mathcal{V}$, i.e.\ also a locally defined spinor bundle $\Sigma \mathcal{V}$. Applying the above discussion pointwise we conclude that
\begin{align}\label{IsoSpinorManifold}
\Sigma M \cong
\begin{cases}
f^{\ast}(\Sigma B) \otimes \Sigma \mathcal{V}, & \text{if $n$ or $k$ is even,}\\
\left( f^{\ast}(\Sigma^+ B) \oplus f^{\ast}(\Sigma^- B) \right) \otimes \Sigma \mathcal{V}, & \text{if $n$ and $k$ are odd.}
\end{cases}
\end{align}

%
%

As we have seen in the last section, the base manifold $B$ and the vertical distribution $\mathcal{V}$, in general, are not spin. Thus, the spinor bundles $\Sigma B$ and $\Sigma \mathcal{V}$ are only defined locally but their tensor product is defined globally. The rules for Clifford multiplication \eqref{CliffordMultCases} carry over to the spinor bundle $\Sigma M$. In the setting of Riemannian manifolds, these rules allow us to distinguish Clifford multiplication with horizontal and vertical vector fields.

\begin{notation}
Let $f: (M,g)  \rightarrow (B,h)$ be a Riemannian submersion such that $M$ has a fixed spin structure. For abbreviation we write
\begin{align*}
\Sigma M \cong f^{\ast} (\SpinB) \otimes \Sigma \mathcal{V},
\end{align*}
where
\begin{align*}
\SpinB \coloneqq
\begin{cases}
\Sigma B, & \text{if $n$ or $k$ is even,}\\
\Sigma^+ B \oplus \Sigma^- B, & \text{if $n$ and $k$ are odd.}
\end{cases}
\end{align*}
\end{notation}


Next we calculate the spinorial connection on $\Sigma M$ with respect to a local orthonormal frame $(\xi_1, \ldots, \xi_n, \zeta_1, \ldots, \zeta_k)$ such that $(\xi_1, \ldots, \xi_n)$ is the horizontal lift of a local orthonormal frame $(\check{\xi}_1, \ldots, \check{\xi}_n)$ in the base space $B$ and $(\zeta_1, \ldots, \zeta_k)$ is a locally defined vertical orthonormal frame . 

We recall that for any vector field $X$ and spinor $\Phi$ the spinorial connection $\nabla^M$ on $M$ is  locally given by
\begin{align*}
\nabla^M_X \Phi = X(\Phi) + \frac{1}{4} \sum_{i,j=1}^{n+k} g(\nabla_X e_i, e_j)\gamma(e_i) \gamma(e_j) \Phi,
\end{align*}
where $(e_1, \ldots, e_{n+k})$ is a local orthonormal frame and the action of the Dirac operator on a spinor $\Phi$ is locally define via
\begin{align*}
D^M\Phi = \sum_{i=1}^{n+k} \gamma(e_i) \nabla^M_{e_i} \Phi.
\end{align*}

It follows from \eqref{IsoSpinorManifold} that locally any spinor $\Phi$ on $M$ can be written as a finite linear combination $\Phi = \sum_{l} f^{\ast}\varphi_l \otimes \nu_l$. The next lemma follows from straightforward calculations, where we use that the Christoffel symbols of $M$ are given by \eqref{ChristoffelSymbols}.

\begin{lem}\label{ConnectionFormula}
Let $f:(M^{n+k},g) \rightarrow (B^n,h)$ be a Riemannian submersion. Suppose that $M$ is a spin manifold with a fixed spin structure. With respect to a local orthonormal frame $(\xi_1, \ldots, \xi_n, \zeta_1, \ldots, \zeta_k)$ as above, any spinor $\Phi = f^{\ast}\varphi \otimes \nu$ satisfies the following identities:
\begin{align*}
\nabla^M_{\xi_{\alpha}}\Phi &= ( f^{\ast} \nabla^B_{\xi_{\alpha}} \varphi )\otimes \nu + f^{\ast}\varphi \otimes \nabla^{\mathcal{V}}_{\xi_{\alpha}} \nu  + \frac{1}{2}  \sum_{\beta =1}^n \gamma(\xi_{\beta})  \gamma \left( A(\xi_{\alpha}, \xi_{\beta}) \right) \Phi  \\
& \eqqcolon \nabla^{\mathcal{T}}_{\xi_{\alpha}}\Phi + \frac{1}{2}  \sum_{\beta =1}^n \gamma(\xi_{\beta}) \gamma \left( A(\xi_{\alpha}, \xi_{\beta}) \right) \Phi, \\ 
\nabla^M_{\zeta_a}\Phi &=  f^{\ast}\varphi \otimes \nabla^Z_{\zeta_a} \nu + \frac{1}{2} \sum_{b=1}^k \gamma( \zeta_b ) \gamma \left( T(\zeta_a, \zeta_b) \right) \Phi + \frac{1}{4} \sum_{\alpha =1}^n \gamma(\xi_{\alpha}) \gamma \left( A(\xi_{\alpha}, \zeta_i) \right) \Phi \\
&\eqqcolon \nabla^Z_{\zeta_a} \Phi + \frac{1}{2} \sum_{b=1}^k \gamma( \zeta_b ) \gamma \left( T(\zeta_a, \zeta_b) \right) \Phi + \frac{1}{4} \sum_{\alpha =1}^n \gamma(\xi_{\alpha}) \gamma \left( A(\xi_{\alpha}, \zeta_a) \right) \Phi,\\
D^M \Phi&= \sum_{\alpha = 1}^n \gamma(\xi_{\alpha}) \nabla^{\mathcal{T}}_{\xi_{\alpha}}\Phi + \sum_{a = 1}^k \gamma(\zeta_a) \nabla^{Z}_{\zeta_a} \Phi  - \frac{1}{2} \sum_{a=1}^k \gamma \left( T(\zeta_a, \zeta_a) \right) \Phi \\
& \ \ \  + \frac{1}{2} \sum_{\substack{\alpha, \beta = 1 \\ \alpha < \beta}}^n \gamma \left( A(\xi_{\alpha}, \xi_{\beta}) \right) \gamma( \xi_{\alpha}) \gamma( \xi_{\beta} ) \Phi\\
&\eqqcolon D^{\mathcal{T}} \Phi + D^Z \Phi  - \frac{1}{2} \sum_{a=1}^k \gamma \left( T(\zeta_a, \zeta_a) \right) \Phi + \frac{1}{2} \gamma( A ) \Phi.
\end{align*}
Here $\nabla^B$, $\nabla^{\mathcal{V}}$ and $\nabla^Z$ are the induced connections by the respective connections on $TM$, defined in Section \ref{Section GeometryAffineFiberBundles}.
\end{lem}

Observe that in the above lemma we introduced the twisted connection  $ \nabla^{\mathcal{T}} = f^{\ast} \nabla^B \otimes \nabla^{\mathcal{V}}$ on $f^{\ast} \SpinB \otimes \Sigma \mathcal{V}$. Moreover, $D^{\mathcal{T}}$ and $D^Z$ are the corresponding Dirac operators for  $\nabla^{\mathcal{T}}$, resp.\ $\nabla^Z$ and 
\begin{align*}
\gamma(A) \coloneqq  \sum_{\substack{\alpha, \beta = 1 \\ \alpha < \beta}}^n \gamma \left( A(\xi_{\alpha}, \xi_{\beta}) \right) \gamma( \xi_{\alpha}) \gamma( \xi_{\beta} )
\end{align*} 
is a symmetric operator acting on spinors. 

\subsection{Spin structures with affine parallel spinors} \label{SectionAffineParallelSpinors}

Let $f: (M,g)  \rightarrow (B,h)$ be a Riemannian affine fiber bundle with infranil fiber $Z$. We set $n \coloneqq \dim(B)$ and $k \coloneqq \dim(Z)$. Further, we assume that $(M,g)$ has a fixed spin structure.  Recall from the beginning of Section \ref{SectionInducedStructures} that there is an induced spin structure on the fibers $Z_p$, $p \in B$ and that all these spin structures are equivalent. Since the induced metric $\hat{g}_p$ is affine parallel for all $p \in B$  the affine connection $\nabla^{\aff}$ induces an affine connection, also denoted by $\nabla^{\aff}$ on $\Sigma M$. The main result of this section is that the space of affine parallel $L^2$-spinors,
\begin{align*}
\mathcal{S}^{\aff} \coloneqq \lbrace \varphi \in L^2(\Sigma M) : \nabla^{\aff} \varphi = 0 \rbrace
\end{align*}
is isometric to a Clifford bundle over $B$.

For any $p \in B$ we consider the space of affine parallel spinors restricted of each fiber $Z_p$, i.e.\
\begin{align*}
\mathcal{P}_p \coloneqq \lbrace \varphi \in L^2(\Sigma Z_p) : \nabla^{\aff} \varphi = 0 \rbrace.
\end{align*}
Here $\Sigma Z_p$ is the spinor bundle of the induced spin structure on $Z_p$. Since the spin structures on the fibers $Z_p$ and $Z_q$ are equivalent for all $p,q \in B$, it follows that the spaces $\mathcal{P}_p$ and $\mathcal{P}_q$ are isomorphic to each other. Moreover, $\mathcal{P}_p$ is finite dimensional \cite[p.\ 183]{LottDirac}, \cite[Appendix A]{RoosDiss}. Let $\mathcal{P} \coloneqq \bigsqcup_{p \in B} \mathcal{P}_p \rightarrow B$ be the associated quasi bundle.

\begin{lem}\label{QIsometry}
Let $f:(M,g) \rightarrow (B,h)$ be a Riemannian affine fiber bundle such that $M$ has a fixed spin structure. Then there is an isometry
\begin{align*}
Q: L^2( \SpinB \otimes  \mathcal{P})&\rightarrow \mathcal{S}^{\aff}.
\end{align*}
Moreover, the connection $\nabla^{\mathcal{T}}$ induces a connection $\widecheck{\nabla}^{\mathcal{T}}$ on $\SpinB \otimes  \mathcal{P}$. Taking a split orthonormal frame $(\xi_1,\ldots, \xi_n,\zeta_1,\ldots, \zeta_k)$ around $x$, any $\check{\Phi} \in L^2(\SpinB \otimes \mathcal{P})$ satisfies
\begin{align*}
\nabla^M_{\xi_{\alpha}}Q(\check{\Phi})_x &= Q(\widecheck{\nabla}^{\mathcal{T}} \check{\Phi})_x + \frac{1}{2} \sum_{\beta=1}^n \gamma(\xi_{\beta}) \gamma(A (\xi_{\alpha}, \xi_{\beta} ) ) Q(\check{\Phi})_x \\
& \ \ \ - \frac{\check{\xi}_{\alpha}(\vol(Z_{f(x)}))}{2 \vol(Z_{f(x)})} Q(\check{\Phi})_x.
\end{align*}
\end{lem}

\begin{proof}
It follows from the discussion in Section \ref{SectionSpinSubmersion} that $\Sigma M \cong f^{\ast}(\SpinB) \otimes \Sigma \mathcal{V}$, where the spinor bundles $\SpinB$ and $\Sigma \mathcal{V}$ are in general only defined locally. Let $f^{-1}(U) \cong U \times Z$ be a trivializing neighborhood on which $\SpinB$ and $\Sigma \mathcal{V}$ are well-defined. Then any spinor $\Phi$ restricted to $f^{-1} (U)$ can be written as finite linear combination $\Phi = \sum_{l} f^{\ast}\varphi_l \otimes \nu_l$. By linearity it suffices to consider elementary tensors $f^{\ast}\varphi \otimes \nu$. We observe that
\begin{align*}
\nabla^{\aff}(f^{\ast} \varphi \otimes \nu) = f^{\ast} \varphi \otimes \nabla^{\aff} \nu.
\end{align*}
Thus, a spinor is affine parallel if and only if $\nu$, restricted to any fiber $f^{-1}(p) = Z_p$, is an affine parallel spinor of $Z_p$. Hence, on any trivializing neighborhood we obtain an isomorphism
\begin{align*}
Q^{-1}: \mathcal{S}^{\aff} &\rightarrow L^2(\SpinB \otimes  \mathcal{P}), \\
(f^{\ast} \varphi \otimes \nu)_x &\mapsto  \sqrt{\vol (Z_{f(x)})} \  \varphi_{f(x)} \otimes \nu_{f(x)}, 
\end{align*}
where $\nu_{f(x)}$ denotes the restriction of $\nu$ to the fiber $Z_{f(x)}$. Due to the factor $\sqrt{\vol (Z_{f(x)})}$ it is evident that $Q^{-1}$ defines an isometry. Here $\vol(Z_{f(x)})$ is the volume of the fiber $Z_{f(x)}$ with respect to the induced affine parallel metric $\hat{g}_{f(x)}$. Since the structure group of the Riemannian affine fiber bundle $f: (M,g) \rightarrow (B,h)$ lies in $\Aff(Z)$, $Q^{-1}$ extends to a well-defined global isometry. Taking its inverse gives the desired map $Q$.

It follows from the discussion in Section \ref{Section GeometryAffineFiberBundles} that $\mathcal{S}^{\aff}$ is invariant under the action of $\nabla^{\mathcal{T}}$. Hence, there is an induced connection $\widecheck{\nabla}^{\mathcal{T}}$ on $\SpinB \otimes \mathcal{P}$. Using Lemma \ref{ConnectionFormula} the claimed identity follows from a straightforward calculation.
\end{proof}

\begin{cor}\label{QDirac}
Let $f: (M,g) \rightarrow (B,h)$ be a Riemannian affine fiber bundle with a fixed spin structure on $M$ such that $\mathcal{S}^{\aff}$ is nontrivial. With respect to a split orthonormal frame $(\xi_1, \ldots, \xi_n, \zeta_1, \ldots, \zeta_k)$, any spinor $\Phi \in \mathcal{S}^{\aff}$ satisfies
\begin{align*}
D^M\Phi &= Q \circ \check{D}^{\mathcal{T}} \circ Q^{-1} \Phi + \frac{1}{2}\sum_{\substack{a,b,c=1\\ b < c}}^k \Gamma_{ab}^c \gamma(\zeta_a) \gamma(\zeta_b) \gamma(\zeta_c)  \Phi + \frac{1}{2} \gamma( A ) \Phi\\
&\eqqcolon Q \circ \check{D}^{\mathcal{T}} \circ Q^{-1} \Phi + \frac{1}{2}\gamma(\mathcal{Z}) \Phi + \frac{1}{2} \gamma(A) \Phi.
\end{align*}
Here, $\check{D}^{\mathcal{T}}$ is the Dirac operator on $\SpinB \otimes \mathcal{P}$ associated to the connection $\check{ \nabla}^{\mathcal{T}}$.
\end{cor}

\begin{proof}
Let $\Phi$ be an affine parallel spinor. Since $f: (M,g) \rightarrow (B,h)$ is a Riemannian affine fiber bundle, $\mathcal{S}^{\aff}$ is invariant under the action of the Dirac operator $D^M$. With respect to a split orthonormal frame $(\xi_1, \ldots, \xi_n, \zeta_1, \ldots, \zeta_k)$, we obtain
\begin{align*}
D^M \Phi &= \sum_{\alpha = 1}^n \gamma(\xi_{\alpha}) \nabla_{\xi_{\alpha}} \Phi + \sum_{i=1}^k \gamma(\zeta_a) \nabla_{\zeta_a} \Phi \\
&= \sum_{\alpha = 1}^n \gamma(\xi_{\alpha}) \nabla_{\xi_{\alpha}} \left(Q \circ Q^{-1} (\Phi) \right) + \frac{1}{4}\sum_{a=1}^k \gamma(\zeta_a) \left( \sum_{b,c=1}^k \Gamma_{ab}^c\gamma(\zeta_b)\gamma(\zeta_c) \Phi \right)\\
& \ \ \ - \frac{1}{2}\sum_{a=1}^k \gamma(T(\zeta_a, \zeta_a) )\Phi + \frac{1}{4}\sum_{a=1}^k \sum_{\alpha=1}^n \gamma(\zeta_a) \gamma(\xi_{\alpha}) \gamma(A(\xi_{\alpha},\zeta_a) )\Phi\\
&=  Q \circ \check{D}^{\mathcal{T}} \circ Q^{-1} \Phi + \frac{1}{2} \sum_{\alpha =1}^n\sum_{\beta=1}^n \gamma(\xi_{\alpha})\gamma(\xi_{\beta}) \gamma(A (\xi_{\alpha}, \xi_{\beta} ) ) \Phi\\
& \ \ \ - \frac{1}{2}\gamma\left(\frac{\grad(\vol(Z_{f(x)})}{\vol(Z_{f(x)})} \right)\Phi + \frac{1}{2}\sum_{\substack{a,b,c=1\\ b < c}}^k \Gamma_{ab}^c \gamma(\zeta_a) \gamma(\zeta_b) \gamma(\zeta_c)  \Phi \\
& \ \ \ - \frac{1}{2}\sum_{a=1}^k \gamma(T(\zeta_a, \zeta_a) )\Phi - \frac{1}{4}\sum_{\alpha=1}^n \sum_{ \beta=1}^n  \gamma(\xi_{\alpha}) \gamma(\xi_{\beta}) \gamma(A(\xi_{\alpha},\xi_{\beta}) )\Phi \\
&=  Q \circ \check{D}^{\mathcal{T}} \circ Q^{-1} \Phi + \frac{1}{2}\gamma(\mathcal{Z}) \Phi + \frac{1}{2} \gamma(A) \Phi.
\end{align*}
Here we used Lemma \ref{ConnectionFormula}, Lemma \ref{QIsometry}, and the formulas for the Christoffel symbols,  \eqref{ChristoffelSymbols}. The last line follows from the relation (see for instance \cite[Lemma 1.17.2]{GilkeySubmersion}) 
\begin{align*}
\sum_{a=1}^k T(\zeta_a,\zeta_a) = -\grad (\ln (\vol(Z_p))).
\end{align*}
between the mean curvature $\sum_{a=1}^k T(\zeta_a,\zeta_a)$ and the change of the volume of the fiber $Z_p$. 
\end{proof}

\section{The Dirac operator under collapse to a smooth limit space }\label{SectionMainTheorem}

In \cite{LottDirac} Lott studied the behavior of Dirac eigenvalues on arbitrary collapsing sequences in $\mfdspace$. There Lott combined his results for the eigenvalues of the $p$-form Laplacian on collapsing sequences \cite{LottLaplaceSmooth, LottLaplaceSingular} with the  Bochner-type formulas for the Dirac operator. Moreover, Lott's results also hold for the Dirac operator on differential forms, i.e.\ the operator
\begin{align*}
D =\mathrm{d} + \mathrm{d}^{\ast}: \Omega^{\ast}(M) \rightarrow \Omega^{\ast}(M),
\end{align*}
where $\mathrm{d}^{\ast}$ is the adjoint of the exterior derivative $\mathrm{d}$ with respect to the $L^2$-inner product.

\begin{thm}[{\cite[Theorem 3]{LottDirac}}]\label{LottSingular}
Given $n \in \bbN$ and $G \in \lbrace \SO(n), \Spin(n) \rbrace$, let $(M_i, g_i)_{i \in \bbN}$ be a sequence of connected closed oriented $n$-dimensional Riemannian manifolds with a $G$-structure. Let $V$ be a $G$-Clifford module. Suppose that for some $d, K > 0$ and for each $i \in \bbN$ we have $\diam(M_i) \leq d$ and $\Vert R^{M_i}\Vert_{\infty} \leq K$. Then there is
\begin{enumerate}
\item a subsequence of $(M_i,g_i)_{i \in \bbN}$ which we relabel as $( M_i )_{i \in \bbN}$,
\item a smooth closed $G$-manifold $\check{X}$ with a $G$-invariant Riemannian metric $g^{T\check{X}}$ which is $C^{1,\alpha}$-regular for all $\alpha \in  [0,1)$,
\item a positive $G$-invariant function $\chi \in C(\check{X})$ with $\int_{\check{X}} \chi \dvol =1$,
\item a $G$-invariant function $\mathcal{V} \in L^{\infty}(\check{X}) \otimes \End(V)$ such that if $\Delta^{\check{X}}$ denotes the Laplacian on $L^2(\check{X}, \chi \dvol) \otimes V$ and $\vert D^X \vert$ denotes the operator $\sqrt{\Delta^{\check{X}} + \mathcal{V}}$ acting on the $G$-invariant subspace $(L^2(\check{X}, \chi \dvol) \otimes V)^G$ then for all $k \in \bbN$,
\begin{align*}
\lim_{i \rightarrow \infty} \lambda_k (\vert D^{M_i} \vert ) = \lambda_k( \vert D^X \vert ).
\end{align*}
\end{enumerate}
\end{thm}

The limit measure $\chi$ is also necessary for the analogous results regarding the behavior of the eigenvalues of the Laplacian on functions \cite{FukayaLaplace} and the eigenvalues of the Laplacian on forms \cite{LottLaplaceSmooth, LottLaplaceSingular}, see Example \ref{ExampleEigenvalue}. 

For the special case of collapsing sequences in $\mfdspace$ with a smooth limit space Lott proved an accentuation of the above theorem \cite[Theorem 2]{LottDirac}. In the following we say that for a given $\eps > 0$ two collections of real numbers $(a_i)_{i \in I}$ and $(b_j)_{j \in J}$ are \textit{$\eps$-close} if there is a bijection $\alpha: I \rightarrow J$ such that $\vert b_{\alpha(i)} - a_i \vert \leq \eps$ holds for all $i \in I$.

\begin{thm}[{\cite[Theorem 2]{LottDirac}}]\label{LottSmooth}
Let $B$ be a fixed smooth connected closed Riemannian manifold and let $n \in \bbN$,  $G \in \lbrace \SO(n), \Spin(n) \rbrace$ and  $V$ be a $G$-Clifford module. For any $\eps > 0$ and $K > 0$, there are positive constants $A(B,n,V,\eps,K), A^{\prime}(B,n,V,\eps,K)$ and $C(B,n,V,\eps,K)$ such that the following holds. Let $M$ be an $n$-dimensional connected closed oriented Riemannian manifold with a $G$-structure such that $\Vert R^M \Vert_{\infty} \leq K$ and $\dGH(M,B) \leq A^{\prime}$. Then there are a Clifford module $E^B$ on $B$ and a certain first order differential operator $\mathcal{D}^B$ on $C^{\infty}(B; E^B)$ such that
\begin{enumerate}
\item $\lbrace \arsinh (\frac{\lambda}{\sqrt{2K}} : \lambda \in \sigma(D^M), \ \lambda^2 \leq A \dGH(M,B)^{-2} - C\rbrace$ is $\eps$-close to a subset of $\lbrace \arsinh (\frac{\lambda}{\sqrt{2K}} : \lambda \in \sigma (\mathcal{D}^B) \rbrace$,
\item $\lbrace \arsinh (\frac{\lambda}{\sqrt{2K}} : \lambda \in \sigma(\mathcal{D}^B), \ \lambda^2 \leq A \dGH(M,B)^{-2} - C\rbrace$ is $\eps$-close to a subset of $\lbrace \arsinh (\frac{\lambda}{\sqrt{2K}} : \lambda \in \sigma (D^M) \rbrace$.
\end{enumerate}
\end{thm}

The main result, Theorem \ref{ThmIntroduction}, deals with the special case of collapsing sequences $(M_i, g_i)_{i \in \bbN}$ of spin manifolds in $\mfdspace$ converging to a Riemannian manifold $(B,h)$. There we will give an explicit description of the limit operator $\mathcal{D}^B$ as a twisted Dirac operator with a symmetric $C^{0,\alpha}$-potential. Moreover, we show that the limit operator $\mathcal{D}^B$ satisfies the conclusions of Theorem \ref{LottSingular} with $\chi \equiv 1$. This is a special behavior for the eigenvalues of the Dirac operator on spin manifolds and does not extend to the eigenvalues of the Dirac operator acting on differential forms. 

\begin{ex}\label{ExampleEigenvalue}
Consider the torus $\bbT^2 = \lbrace ( e^{is}, e^{it}) : s,t \in \bbR \rbrace$ with the Riemannian metric
\begin{align*}
g_{\eps} \coloneqq \mathrm{d}s^2 +
\eps^2 c(s)^2 \mathrm{d}t^2,
\end{align*}
for a fixed positive function $c: S^1 \rightarrow \bbR_+$. The sequence $ (\bbT^2, g_{\eps} )_{\eps > 0}$ is a collapsing sequence with bounded sectional curvature and diameter converging to $(S^1, \mathrm{d}s^2 )$ in the Gromov-Hausdorff topology as $\eps \rightarrow 0$. We observe that the integrability tensor $A_{\eps}$ vanishes identically for all $\eps$ and the $T$-tensor is characterized by $\frac{c^{\prime}(s)}{c(s)}$ independent of $\eps$. In particular, the $T$-tensor is nontrivial for all $\eps$ if we choose the function $c$ to be non constant. 

We consider  the Dirac operator acting on differential forms. In that case the space of affine parallel forms is given by
\begin{align*}
\mathcal{S}^{\aff} \coloneqq \lbrace f \in C^{\infty}(T^2) : \frac{\partial}{\partial t} f = 0 \rbrace \cup \lbrace \alpha \, \mathrm{d}s \in \Omega^1(T^2) : \frac{\partial}{\partial t} \alpha = 0 \rbrace.
\end{align*}
The Dirac operator $D_{\eps} = \mathrm{d} + \mathrm{d}^{\ast}$ acts on $(f + \alpha \mathrm{d}s) \in \mathcal{S}^{\aff}$ as 
\begin{align*}
D_{\eps}(f(s) + \alpha(s) \mathrm{d}s) = \frac{\partial}{\partial s} f(s) \mathrm{d}s - c(s)^{-1} \frac{\partial}{\partial s}(c(s) \alpha(s) ). 
\end{align*}
We observe that $\left(D_{\eps}\right)_{\vert \mathcal{S}^{\aff}}$ is independent of $\eps$. In particular, in the limit $\eps \rightarrow 0$, the sequence $\left(D_{\eps}\right)_{\vert \mathcal{S}^{\aff}}$ induces a first order differential operator $D_0$ on $\Omega^{\ast}(S^1)$. For any eigenform $f(s) + \alpha(s) \mathrm{d}s  \in \Omega^1(S^1)$ of $D_0$ with eigenvalue $\lambda$ we have that
\begin{align*}
\lambda f(s) &= - c(s)^{-1} \frac{\partial}{\partial s}(c(s) \alpha(s) ), \\
\lambda \alpha(s) \mathrm{d}s &=  \frac{\partial}{\partial s} f(s) \mathrm{d}s.
\end{align*}


For example, if $c(s) = e^{\cos(s)}$ then one can calculate numerically, adapting the algorithm from \cite[p.\ 3 - 6]{Strohmaier}, that, without counting multiplicities, the first eigenvalues are approximately given by
\begin{align*}
\lambda_0 = 0, \; \lambda_1 \approx 0,990 , \; \lambda_2 \approx 1,137.
\end{align*}
In particular, the spectrum of $D_0$ does not coincide with the spectrum of the Dirac operator $D^{S^1}$ as $
\sigma(D^{S^1}) = \bbZ$.
\end{ex}

Before we prove Theorem \ref{ThmIntroduction} we need to fix some notations. Let $(M_i,g_i)_{i \in \bbN}$  be a collapsing sequence of Riemannian spin manifolds in $\mfdspacenk$ converging to an $n$-dimensional Riemannian manifold $(B,h)$. By Theorem \ref{ThmInvariantMetrics} there is an $I \in \bbN$ such that for any $i \geq I$ there are metrics $\tilde{g}_i$ on $M_i$ and $\tilde{h}_i$ on $B$ such that $f_i: (M_i, \tilde{g}_i) \rightarrow (B,\tilde{h}_i)$ is a Riemannian affine fiber bundle. Moreover,
\begin{align*}
\lim_{i \rightarrow \infty} \Vert \tilde{g}_i - g_i \Vert_{C^1} &= 0, \\ 
\lim_{i \rightarrow \infty} \Vert \tilde{h}_i - h \Vert_{C^1} &= 0.
\end{align*}
Let $\Sigma M_i$ and $\widetilde{\Sigma M}_i$ be the spinor bundles of $(M_i,g_i)$ and $(M_i, \tilde{g}_i)$ respectively. There is an explicit isometry \cite[Section 2.2]{Maier}.
\begin{align}\label{SpinBundleIsometry}
\Theta_i: L^2(\Sigma M_i) \rightarrow L^2(\widetilde{\Sigma M}_i ).
\end{align}

For a Riemannian affine fiber bundle $f_i: (M_i, \tilde{g}_i) \rightarrow (B,\tilde{h}_i)$  the space of affine parallel spinors $\mathcal{S}^{\aff}_i \subset L^2(\widetilde{\Sigma M}_i)$ is well-defined. The Dirac operator $\tilde{D}^{M_i}$ on $(M_i, \tilde{g}_i)$ acts diagonally with respect to the splitting
\begin{align*}
L^2(\widetilde{\Sigma M}_i) = \mathcal{S}_i^{\aff} \oplus \left( \mathcal{S}^{\aff}_i \right)^{\perp}
\end{align*}
since it commutes with the affine connection $\nabla^{\aff}$. In general, for the original fibration $f_i: (M_i, g_i) \rightarrow (B,h)$ the induced metrics on the fibers is not affine parallel. Thus, the affine connection $\nabla^{\aff}$ does not induce a well-defined connection on the spinor bundle $\Sigma M_i$. Instead we use the isometry $\Theta_i$ to define
\begin{align*}
\mathcal{S}_i \coloneqq \Theta_i^{-1}(\mathcal{S}_i^{\aff}).
\end{align*}
This induces the splitting
\begin{align*}
L^2(\Sigma M_i) = \mathcal{S}_i \oplus \mathcal{S}_i^{\perp}.
\end{align*}
But in contrast to Riemannian affine fiber bundles, the Dirac operator $D^{M_i}$ on $(M_i, g_i)$, in general, does not act diagonally with respect to this splitting. Nevertheless, it follows from the continuity of the spectra of Dirac operators \cite[Main Theorem 2]{Nowaczyk} that  the spectra of the restrictions of $D^{M_i}$ and $\tilde{D}^{M_i}$ to $\mathcal{S}_i$ and $\mathcal{S}_i^{\aff}$, respectively to their orthogonal complements, have the same limit as $i \rightarrow \infty$. 

%

Similar to \eqref{SpinBundleIsometry} there is also an isometry
\begin{align*}
\theta_i: L^2(TM_i) \rightarrow L^2(\widetilde{TM}_i).
\end{align*}
In the next theorem we interpret the operators $\nabla^{\mathcal{V}_i}$, $\mathcal{Z}_i$ and $\mathcal{A}_i$ on $(M_i, g_i)$ as the pullbacks of the respective operators on $(M_i, \tilde{g}_i)$, introduced in Section \ref{SectionOperatorsRiemannianAffine},  via the map $\theta_i$.

Using this terminology we state the explicit description of the limit operator $\mathcal{D}^B$ in Theorem \ref{LottSmooth} for collapsing sequences of spin manifolds converging to a Riemannian manifold of lower dimension in the following theorem from which Theorem \ref{ThmIntroduction} follows immediately.

\begin{thm}\label{MainTheorem}
Let $(M_i,g_i)_{i \in \bbN}$ be a sequence of spin manifolds in~$\mfdspacenk$ converging to a smooth $n$-dimensional Riemannian manifold $(B,h)$ such that the space $\mathcal{S}_i$ is nontrivial  for almost all $i \in \bbN$. Then there is a subsequence $(M_i,g_i)_{i \in \bbN}$ such that the spectrum of $D^{M_i}_{\vert \mathcal{S}_i}$ converges to the spectrum of the elliptic self-adjoint first order differential operator
\begin{align*}
\mathcal{D}^B: \dom(\mathcal{D}^B) &\rightarrow L^2(\SpinB \otimes \mathcal{P} ), \\
\Phi &\mapsto \check{D}^{\mathcal{T}_{\infty}} \Phi + \frac{1}{2} \gamma(\check{\mathcal{Z}}_{\infty}) \Phi + \frac{1}{2}\gamma(\check{\mathcal{A}}_{\infty}) \Phi,
\end{align*}
where
\begin{align*}
\SpinB \coloneqq
\begin{cases}
\Sigma B, & \text{if $n$ or $k$ is even,}\\
\Sigma^+ B \oplus \Sigma^- B, & \text{if $n$ and $k$ are odd.}
\end{cases}
\end{align*}
Further, 
\begin{enumerate}
\item $\mathcal{P}$ represents the space of affine parallel spinors of the fibers,

\item $\check{D}^{\mathcal{T}_{\infty}}$ is the twisted Dirac operator on  $\SpinB \otimes \mathcal{P}$ with respect to the twisted connection $\widecheck{\nabla}^{\mathcal{T}_{\infty}} = \nabla^h \otimes \nabla^{\mathcal{V}_{\infty}}$, where $\nabla^h$ is the spinorial connection on $(B,h)$ and $\nabla^{\mathcal{V}_{\infty}}$ is induced by the $C^{0,\alpha}$-limit of $\nabla^{\mathcal{V}_i}$ for any $\alpha \in [0,1)$,

\item $\check{\mathcal{Z}}_{\infty}$ is induced by the $C^{0,\alpha}$-limit of $(\mathcal{Z}_i)_{i \in \bbN}$ for any $\alpha \in [0,1)$,

\item $\check{\mathcal{A}}_{\infty}$ is the $C^{0,\alpha}$-limit of $(\mathcal{A}_i)_{i \in \bbN}$ for any $\alpha \in [0,1)$.
\end{enumerate}
\end{thm}

\begin{proof}

%
%

By the discussion above, we can assume without loss of generality that $( f_i: (M_i, g_i) \rightarrow (B,h_i) )_{i \in \bbN}$ is a collapsing sequence of Riemannian affine fiber bundles converging to $(B,h)$.

In this setting we can use the isometry $Q_i: L^2(\SpinBi \otimes \mathcal{P}) \rightarrow \mathcal{S}^{\aff}_i$, see Lemma \ref{QIsometry}, and apply Corollary \ref{QDirac} to write
\begin{align*}
D^{M_i} = Q_i \circ \check{D}^{\mathcal{T}_i} \circ Q_i^{-1} + \frac{1}{2} \gamma(\mathcal{Z}_i) + \frac{1}{2} \gamma(\mathcal{A}_i).
\end{align*}
From the discussion in Section \ref{SectionOperatorsRiemannianAffine} it follows that $\gamma(\mathcal{Z}_i)$ and $\gamma(\mathcal{A}_i)$ act diagonally with respect to the splitting $L^2(\Sigma M_i) = \tilde{\mathcal{S}}_i^{\aff} \oplus  \left( \tilde{\mathcal{S}}_i^{\aff} \right)^{\perp}$. Thus, there are well-defined operators $\check{\mathcal{Z}}_i$ and $\check{\mathcal{A}}_i$ such that
\begin{align*}
\left. D^{M_i} \right._{\vert\tilde{\mathcal{S}}^{\aff} } &= Q_i \circ \left( \check{D}^{\mathcal{T}_i} + \frac{1}{2} \gamma(\check{\mathcal{Z}}_i) + \frac{1}{2} \gamma(\check{\mathcal{A}}_i) \right) \circ Q_i^{-1} \\
& \eqqcolon Q_i \circ \mathfrak{D}_i \circ Q_i^{-1}.
\end{align*}
Since $Q_i: L^2(\SpinBi \otimes \mathcal{P}_i) \rightarrow \tilde{\mathcal{S}}^{\aff}$ is an isometry, it follows that $D^{M_i}_{\vert\tilde{\mathcal{S}}^{\aff} }$ is isospectral to $\mathfrak{D}_i$. Furthermore, the operator $\mathfrak{D}_i$ is densely defined on $H^{1,2}(\SpinBi \otimes \mathcal{P}_i)$. 

Now, we are going to choose a subsequence such that the spectrum $\sigma( \mathfrak{D}_i)$ converges to the spectrum of the claimed operator $\mathcal{D}^B$. The main problem here is that the operators $\mathfrak{D}_i$ are defined on different spaces. Thus, we need to find a common space on which we can study the behavior of the spectrum of the sequence $(\mathfrak{D}_i)_{i \in \bbN}$. This is done in the next three steps. 

By abuse of notation, we use the same index $i$ for any subsequence we choose. The following identifications are based on the constructions in \cite[Section 4]{LottDirac}. The idea here is similar to Fukaya's main idea in \cite{FukayaBoundary}. Namely, we will consider the corresponding sequence of $\Spin(n+k)$-principal bundles and identify the spinors on $(M_i, g_i)$ with $\Spin(n+k)$-invariant functions on the corresponding $\Spin(n+k)$-principal bundle $P_{\Spin}(M_i, \tilde{g}_i)$.
\newline

\textit{Step 1: Identification of the spinors.} Let $P_i \coloneqq P_{\Spin}(M_i, g_i)$ be the $\Spin(n+k)$-principal bundle of $(M_i, g_i)$. Further, we set $g_i^P$ to be a Riemannian metric on $P_i$ such that
\begin{align*}
\pi_i: (P_i, g_i^P) \rightarrow (M_i, g_i)
\end{align*}
is a Riemannian submersion with totally geodesic fibers and $\vol(\pi_i^{-1}(x) )= 1$ for all $x \in M_i$.

The isometric $\Spin(n+k)$ action on $P_i$ together with the canonical complex spinor representation $\theta_{n+k}: \Spin(n+k) \rightarrow \Sigma_{n+k}$ induces an isometric $\Spin(n+k)$ action on the tensor product $L^2(P_i, g_i^P) \otimes \Sigma_{n+k}$.  Let $\left( L^2(P_i, g_i^P) \otimes \Sigma_{n+k} \right)^{\Spin(n+k)}$ be the subspace that is invariant under this $\Spin(n+k)$ action.  

Since the spinor bundle $\Sigma M_i$ of $(M_i, g_i)$ is defined as $\Sigma M_i = P_i \times_{\theta_{n+k}} \Sigma_{n+k}$ it follows at once that there is a canonical isomorphism
\begin{align}\label{FirstIsometry}
\Pi_i: \left( L^2(P_i, g_i^P) \otimes \Sigma_{n+k} \right)^{\Spin(n+k)} \rightarrow L^2(\Sigma M_i ).
\end{align}


\medskip

\textit{Step 2: Convergence of the $\Spin(n+k)$-principal bundles.} Next we consider the sequence $(P_i, g_i^P )_{i \in \bbN}$. By construction, the sectional curvatures and the diameter of this sequence are uniformly bounded in $i$. Thus, we can apply the $G$-equivariant version of Gromov's compactness theorem, \cite[Lemma 1.11 and Lemma 1.13]{FukayaBoundary} to the sequence $(P_i, g_i^P )_{i \in \bbN}$. It follows that there is a subsequence, which we denote again by $(P_i, g_i^P )_{i \in \bbN}$, that converges to a compact metric space $(\tilde{B}, h^P)$ on which $\Spin(n+k)$ acts as isometries. In particular, $\faktor{(\tilde{B}, h^P)}{\Spin(n+k)}$ is isometric to the limit space $(B, h)$ of the sequence $(M_i, g_i)_{i \in \bbN}$. Using the same strategy as in \cite[Theorem 6.1]{FukayaBoundary} it follows that $(\tilde{B}, h^P)$ is a Riemannian manifold with a $C^{1, \alpha}$-metric.

As Fukaya's fibration theorem also holds in a $G$-equivariant setting \cite[Theorem 9.1]{FukayaBoundary}, it follows that there is a further subsequence $(P_i, g_i^P)_{i \in \bbN}$ such that for all $i \in \bbN$ there is a $\Spin(n+k)$-equivariant fibration $\tilde{f}_i: P_i \rightarrow \tilde{B}$ with infranil fibers and affine structure group. Since for every $i \in \bbN$, the metric $g_i$ on $M_i$ is invariant, it follows that $g_i^P$ is also an invariant metric, i.e.\ there is a $\Spin(n+k)$-invariant metric $h_i^P$ on $\tilde{B}$ such that 
\begin{align*}
\tilde{f}_i: (P_i, g_i^P) \rightarrow (\tilde{B}, h_i^P)
\end{align*}
is a $\Spin(n+k)$-equivariant Riemannian affine fiber bundle. Moreover, the following diagram commutes for every $i \in \bbN$,
\begin{equation*}
\begin{tikzcd}
(P_i, g_i^P) \arrow[d] \arrow[r , "\tilde{f}_i"] & (\tilde{B}, h_i^P) \arrow[d] \\
(M_i, g_i) \arrow[r, "f_i"] & (B, h_i) \rlap{ .}
\end{tikzcd}
\end{equation*}

\medskip

\textit{Step 3: The space of affine parallel spinors.} We fix an $i \in \bbN$ and recall the affine connection $\nabla^{\aff}$ on $(M_i,h_i)$ that is induced by the affine connection on the infranil fiber $Z_i$ of the Riemannian affine fiber bundle $f_i: (M_i, g_i) \rightarrow (B, h_i)$. Since for all $p \in B$ the induced metric $\hat{g}_p$ on the fiber $Z_p = f^{-1}_i(p)$ is affine parallel, the connection $\nabla^{\aff}$ induces an affine connection, on $P_i$. Hence, there is a well-defined subspace $L^2(P_i, g_i^P)^{\aff} \subset L^2(P_i, g_i^P)$ consisting of affine parallel functions. As we already know that the space $\left( L^2(P_i, g_i^P) \otimes \Sigma_{n+k}\right)^{\Spin(n+k)}$ is isometric to $L^2(\Sigma M_i)$, see \eqref{FirstIsometry},  it follows that there is an induced isometry
\begin{align}\label{FirstAffineIsometry}
\widetilde{\Pi}_i: \left( L^2(P_i, g_i^P)^{\aff} \otimes \Sigma_{n+k} \right)^{\Spin(n+k)} \rightarrow \mathcal{S}_i^{\aff}.
\end{align}
Here $\mathcal{S}_i^{\aff}$ denotes as usual the space of affine parallel spinors.

Next, we observe that we can view $L^2(P_i, g_i^P)^{\aff}$ also as the space of functions in $L^2(P_i, g_i^P)$ that are constant along the fibers of the fibration $\tilde{f}_i: \tilde{P}_i \rightarrow \tilde{B}$. In particular, for any $s \in L^2(P_i, g_i^P)^{\aff}$ there is an $\check{s} \in L^2(\tilde{B}, h_i^P)$ such that $\tilde{f}_i^{\ast} \check{s} = s$. Hence, we have the isometry
\begin{align*}
L^2(\tilde{B}, \tilde{h}_i^P) &\rightarrow L^2(P_i, g_i^P )^{\aff}, \\
\check{s} &\mapsto f_i^{\ast}(v_i^{-\frac{1}{2}} \check{s}),
\end{align*}
where $v_i(p) =  \vol(\tilde{f}_i^{-1}(p))$ for all $p \in \tilde{B}$. Combined with \eqref{FirstAffineIsometry} we obtain the isometry
\begin{align*}
\tilde{Q}_i: \left(L^2(\tilde{B}, h_i^P) \otimes \Sigma_{n+k} \right)^{\Spin(n+k)} \rightarrow \mathcal{S}_i^{\aff},
\end{align*}
such that the following diagram commutes
\begin{equation*}
\begin{tikzcd}
\left(L^2(\tilde{B}, h_i^P) \otimes \Sigma_{n+k} \right)^{\Spin(n+k)} \arrow[d, "\rho_i"'] \arrow[rd, "\tilde{Q}_i"] &[+3 em] \\
L^2(\SpinBi \otimes \mathcal{P}_i)  \arrow[r, "Q_i"'] &\mathcal{S}_i^{\aff} \rlap{ .}
\end{tikzcd}
\end{equation*}
Here we used that there is an isometry $\rho_i$ similar to \eqref{FirstIsometry}.

We recall from Section \ref{SectionSpinSubmersion} that
\begin{align*}
\Sigma_{n+k} \cong \tensor[^{\diamond}]{\Sigma}{_n} \otimes \Sigma_k,
\end{align*}
with
\begin{align*}
\tensor[^{\diamond}]{\Sigma}{_n} \coloneqq 
\begin{cases}
\Sigma_n, &\text{if $n$ or $k$ is even,} \\
\Sigma_n^+ \oplus \Sigma_n^-, &\text{if $n$ and $k$ are odd,}
\end{cases}
\end{align*}
where $\Sigma_n^+$ and $\Sigma_n^-$ are two isomorphic copies of $\Sigma_n$ (compare \eqref{IsoOneEven}, \eqref{IsoOddOdd}). Then
\begin{align*}
\left(L^2(\tilde{B}, h_i^P) \otimes \Sigma_{n+k} \right)^{\Spin(n+k)} &\cong \left(L^2(\tilde{B},h_i^P) \otimes \big(\tensor[^{\diamond}]{\Sigma}{_n} \otimes \Sigma_k  \big) \right)^{\Spin(n+k)} \\
&\cong L^2(\SpinBi \otimes \mathcal{P})
\end{align*}
for some fixed locally defined vector bundle $\mathcal{P}$ over $B$ independent of $i$. In particular, there are isomorphisms $\mathcal{P}_i \rightarrow \mathcal{P}$ for all $i \in \bbN$ and therefore also isometries
\begin{align}\label{NewQiIsometry}
Q_i: L^2(\SpinBi \otimes \mathcal{P}) \rightarrow \mathcal{S}_i^{\aff}.
\end{align}

Now, $\SpinBi$ is the only object left that depends on $i$. To remove also this $i$-dependency, let us assume for the moment that $B$ is a spin manifold.  For any $i \in \bbN$ we consider the isometry
\begin{align*}
\hat{\beta}_{h}^{h_i}: L^2(\Sigma B) \rightarrow L^2(\Sigma_i B),
\end{align*}
that was constructed in \cite[Section 2.2]{Maier}.  Here $\Sigma B$ is the spinor bundle of $(B, h)$ and $\Sigma_i B$ is the spinor bundle of $(B, h_i)$. Going back to the original case, where $B$ is not necessarily spin, we can still apply a local version of the isometry $\hat{\beta}_{h}^{h_i}$ to obtain an isometry
\begin{align*}
\Theta_i: L^2( \SpinB \otimes \mathcal{P}) \rightarrow L^2(\SpinBi \otimes \mathcal{P}).
\end{align*}

\medskip

\textit{Step 4: The convergence of the Dirac eigenvalues.} Similar to the beginning of the proof, the Dirac operator $D^{M_i}$ on $(M_i, g_i)$ restricted to $\mathcal{S}_i^{\aff}$ can be written as
\begin{align*}
\left. D^{M_i} \right._{\vert\tilde{\mathcal{S}}^{\aff} } &= Q_i \circ \left( \check{D}^{\mathcal{T}_i} + \frac{1}{2} \gamma(\check{\mathcal{Z}}_i) + \frac{1}{2} \gamma(\check{\mathcal{A}}_i) \right) \circ Q_i^{-1} \\
& =Q_i \circ \mathfrak{D}_i \circ Q_i^{-1},
\end{align*}
where $Q_i$ is now the isometry \eqref{NewQiIsometry}.

For any $i \in \bbN$ the operator
\begin{align*}
\mathcal{D}_i \coloneqq \Theta_i^{-1} \circ \mathfrak{D}_i \circ \Theta_i
\end{align*}
is isospectral to  $D^{M_i}_{\vert \mathcal{S}_i^{\aff}}$ and densely defined on $H^{1,2}( \SpinB \otimes \mathcal{P})$. By a small abuse of notation, we continue to write
\begin{align*}
\mathcal{D}_i = \check{D}^{\mathcal{T}_i} + \frac{1}{2} \gamma(\check{\mathcal{Z}}_i) + \frac{1}{2} \gamma(\check{\mathcal{A}}_i).
\end{align*}
First, we observe that the $C^1$-norms corresponding to $(B, \tilde{h}_i)$ are all equivalent to the $C^1$-norm on $(B,h)$ as $\lim_{i \in \bbN} \Vert \tilde{h}_i - h \Vert_{C^1} = 0$. By Lemma \ref{BoundFiberCon} it follows that the sequence of operators $\left( \gamma(\check{\mathcal{Z}}_i)\right)_{i \in \bbN}$ is uniformly bounded in $C^1(B,h)$. Further, we conclude from Lemma \ref{BoundATensor} that also the sequence $\left( \gamma(\check{\mathcal{A}}_i) \right)_{i \in \bbN}$ is uniformly bounded in $C^1(B,h)$. Since $C^1 \hookrightarrow C^{0,\alpha}$ is a compact embedding for all $\alpha \in [0,1)$ there is a subsequence such that $\left(\gamma(\check{\mathcal{Z}}_i)\right)_{i \in \bbN}$ and $\left( \gamma(\check{\mathcal{A}}_i) \right)_{i \in \bbN}$ converge to well-defined operators $\gamma(\check{\mathcal{Z}}_{\infty})$ and $\gamma(\check{\mathcal{A}}_{\infty})$ in $C^{0, \alpha}$ for any $\alpha \in [0, 1)$. 

As $\widecheck{\nabla}^{\mathcal{T}_i}$ corresponds to the twisted connection $\nabla^{ \tilde{h}_i} \otimes \widecheck{\nabla}^{\mathcal{V}_i}$ it follows from Lemma \ref{BoundConnV} that there is a further subsequence $(M_i, g_i)_{i \in \bbN}$ such that the corresponding sequence $(\mathcal{D}_i)_{i \in \bbN}$ is a sequence of operators that are densely defined on $H^{1,2}(\SpinB \otimes \mathcal{P})$. Furthermore, the sequence $(\mathcal{D}_i)_{i \in \bbN}$ converges in $B(H^{1,2}(\SpinB \otimes \mathcal{P}), L^2(\SpinB \otimes \mathcal{P}))$ to the claimed limit operator $\mathcal{D}^B$. Here, $B(.,.)$ is the space of bounded linear operators endowed with the operator norm. Thus  the spectra $(\sigma (\mathcal{D}_i))_{i \in \bbN}$ converge to $\sigma(\mathcal{D}^B)$ by the results of \cite[Section 4]{Nowaczyk}.
\end{proof}

\begin{rem}
In \cite[Theorem 5.1 and Theorem 5.2]{RoosDirac} we studied the behavior of Dirac eigenvalues of collapsing sequences of spin manifolds in $\mathcal{M}(n+1,d)$. Assuming the limit space to be smooth, \cite[Theorem 5.1 and Theorem 5.2]{RoosDirac} and Theorem \ref{MainTheorem} are compatible. The statement that the spaces $\mathcal{S}_i$ should be nontrivial for almost all $i \in \bbN$ corresponds to the case of projectable spin structures in \cite{RoosDirac}.  In that case we also obtained in the limit a twisted Dirac operator with a potential related to the $A$-tensor. Moreover, the twist was nontrivial if and only if $\nabla^{\mathcal{V}_i}$ where trivial for almost all $i \in \bbN$. The potential $\frac{1}{2} \gamma(\check{\mathcal{Z}}_{\infty})$ does not appear in the case $k=1$, as the fibers are all one-dimensional.
\end{rem}

As a conclusion of Theorem \ref{MainTheorem}  we can characterize the special case where the spectrum of the limit operator $\mathcal{D}^B$ coincides with the spectrum of the Dirac operator on the manifold $B$ up to multiplicity. We formulate the following corollary for collapsing sequences of Riemannian affine fiber bundles. Combining Theorem \ref{ThmInvariantMetrics} and the $C^1$-continuity of Dirac eigenvalues \cite[Main Theorem 2]{Nowaczyk} this corollary extends to any collapsing sequence in $\mfdspacenk$ with smooth $n$-dimensional limit space.

\begin{cor}\label{CorollaryExplicitDirac}
Let $\left(f_i: (M_i, g_i) \rightarrow (B,h_i) \right)_{i \in \bbN}$ be a collapsing sequence of Riemannian affine fiber bundles such that $(M_i, g_i)_{i \in \bbN}$ is a spin manifold in $\mfdspacenk$ and $B$ is a closed $n$-dimensional manifold. Further, we denote by $Z_i$ the closed $k$-dimensional infranilmanifold which is diffeomorphic to the fibers of $f_i: (M_i, g_i) \rightarrow (B,h_i)$. If
\begin{gather*}
\limsup_{i \rightarrow \infty} \Vert \Hol(\mathcal{V}_i, \nabla^{\mathcal{V}_i}) - \id \Vert_{\infty} = 0,\\
\limsup_{i \rightarrow \infty} \left( \sup_{p \in B} \Vert \scal(Z^i_p) \Vert_{\infty} \right) = 0,\\
\limsup_{i \rightarrow \infty} \Vert A_i \Vert_{\infty} = 0 ,
\end{gather*}
and $\mathcal{S}_i^{\aff}$ is nontrivial for almost all $i \in \bbN$, then there is a subsequence also denoted by $(M_i, g_i)_{i \in \bbN}$ such that the spin structure on $(M_i,g_i)$ induces the same spin structure on $B$ for all $i \in \bbN$ and such that the spectrum of the Dirac operator $D^{M_i}_{\vert \mathcal{S}_i^{\aff}}$ converges,  up to multiplicity, to the spectrum of $D^B$, if $n$ or $k$ is even, and to the spectrum of $D^B \oplus - D^B$, if $n$ and $k$ are odd. Each eigenvalue is counted $\rank(\mathcal{P})$-times.
\end{cor}

\begin{proof}
From the above theorem it follows that the limit operator $\mathcal{D}^B$ equals $D^B \otimes \id$, respectively $(D^B \oplus - D^B) \otimes \id$ if
\begin{enumerate}
	\item  $\nabla^{\mathcal{V}_{\infty}}$ is gauge equivalent to the trivial connection,
	\item $\mathcal{Z}_{\infty} = 0$,
	\item $\mathcal{A}_{\infty} = 0$.
\end{enumerate}
Regarding the first point, we note that  $\nabla^{\mathcal{V}_{\infty}}$ is gauge equivalent to the trivial connection if $(\mathcal{V}_i, \nabla^{\mathcal{V}_i})$ is in the limit $i\rightarrow \infty$ a trivial vector bundle with trivial holonomy, see Remark \ref{RemarkVertHol}. In this case it is immediate that $\mathcal{P}$ is the trivial vector bundle. Hence, $\SpinB$ is globally well-defined. In particular, there is a well-defined induced spin structure on $B$. As there are only finitely many equivalence classes of spin structures  on a fixed closed Riemannian manifold \cite[Chapter II, Theorem 1.7]{LawsonMichelsohn}, we can choose a subsequence, again denoted by $(M_i, g_i)_{i \in \bbN}$ such that the spin structure on $(M_i, g_i)$ induces the same spin structure on $B$ for all $i \in \bbN$. 

Since $\Vert \mathcal{Z}_i \Vert_{\infty} \leq 3 \Vert \scal^{Z_i} \Vert_{\infty}$, see the proof of Lemma \ref{BoundFiberCon}, the second condition implies that the limit $\mathcal{Z}_{\infty}$ vanishes identically. Finally, it is immediate that $\mathcal{A}_{\infty} = 0$ is equivalent to the vanishing of the $A$-tensor in the limit since $\Vert \mathcal{A}_i \Vert_{\infty} = \Vert A_i \Vert_{\infty}$ by definition. 

For the last statement, let $l \coloneqq \rank(\mathcal{P})$. Since $\mathcal{P}$ is the trivial vector bundle there is  a global frame $(\rho_1, \ldots, \rho_l)$. Let $\varphi$ be an eigenspinor of $D^B$, resp.\ $D^B \oplus - D^B$. Then for any $1 \leq j \leq l$ the spinor $\varphi \otimes \rho_j$ is an eigenspinor of $\mathcal{D}^B$ with the same eigenvalue. Hence, any eigenvalue of  $D^B$, resp.\ $D^B \oplus - D^B$ is counted $l$-times.
\end{proof}

We conclude that there are three geometric obstructions for a convergence to the Dirac operator on the base space. As discussed in the Examples \ref{ExampleWithoutHolonomy}, \ref{ExampleWithHolonomy}, \ref{ExampleHeisenberg}, \ref{ExampleNonVanishingA} these geometric obstructions are all independent of each other. In the following example we discuss a class of collapsing sequences that satisfy the assumptions of Corollary \ref{CorollaryExplicitDirac}.

\begin{ex}\label{ExampleMaxTorus}
Let $G$ be a compact $m$-dimensional Lie group with Lie algebra $\mathfrak{g}$. Since $G$ is compact we can choose a biinvariant metric $g$.  Then we fix the spin structure
\begin{align}\label{ExampleGSpin}
P_{\Spin}G \cong G \times \Spin(m).
\end{align}
Next we fix a maximal torus $\bbT^k$ in $G$. The torus $\bbT^k$ acts on $G$ via left multiplication. Since the metric $g$ is biinvariant, the maximal torus $\bbT^k$ acts on $G$ as isometries. In the following we consider the homogeneous space $B \coloneqq \mfaktor{\bbT^k}{G}$ with the induced quotient metric $h$. Let $\mathfrak{g} = \mathfrak{t} + \mathfrak{b}$ be the splitting of the Lie algebra of $G$ into the Lie algebra $\mathfrak{t}$ of $\bbT^k$ and its orthogonal complement $\mathfrak{b} \cong T_{f(e)}B$ with respect to the biinvariant metric $g$. Here $e$ is the neutral element in $G$ and $f: G \rightarrow B$ is the quotient map.


By construction $f: (G,g) \rightarrow (B,h)$ is a $\bbT^k$-principal bundle. Furthermore, $f$ is a Riemannian submersion with totally geodesic fibers. We can write $g = \check{g} + f^{\ast}h$, where $\check{g}$ vanishes on vectors orthogonal to the fibers. For any $\eps > 0$ we define $g_{\eps} \coloneqq \eps^2 \check{g} + f^{\ast} h$. We observe that $g_{\eps}$ is left invariant for all $\eps > 0$ and biinvariant if and only if $\eps = 1$. As $\eps \rightarrow 0$ the sequence $(G,g_{\eps})_{\eps}$ converges to $(B,h)$ in the Gromov-Hausdorff topology. For abbreviation we denote by $G_{\eps}$ the Riemannian manifold $(G, g_{\eps})$. It follows from \cite[Theorem 2.1]{CheegerGromovCollapse1} that there are constants $C$ and $d$ such that $\vert \sec(G_{\eps} ) \vert \leq C$ and $\diam(G_{\eps})\leq d$ for all $\eps \in (0,1)$.

Now we show that the assumptions of Corollary \ref{CorollaryExplicitDirac} are fulfilled. Since the fibers of $f_{\eps}. (G,g_{\eps}) \rightarrow (B,h)$ are totally geodesic for all $\eps > 0$ it follows that the tensor $T_{\eps}$ vanishes identically for all $\eps > 0$. Moreover, the vertical distribution $\mathcal{V}_{\eps}$ is a trivial $\bbR^k$ vector bundle over $G_{\eps}$ as $\bbT^k$ acts on $G_{\eps}$ as isometries for all $\eps > 0$. Thus $\nabla^{\mathcal{V}_{\eps}}$ is gauge equivalent to the trivial connection for all $\eps > 0$. Applying this gauge transformation if necessary, we can assume without loss of generality that $\nabla^{\mathcal{V}_{\eps}}$ is the trivial connection on $\mathcal{V}_{\eps}$ for all $\eps > 0$. By construction, the fibers of $f_{\eps}$ are embedded flat tori. Thus, the induced Levi-Civita connection on the fiber is the affine connection, i.e.\ $\mathcal{Z}_{\eps} = 0$ for all $\eps > 0$. Next, we take a global orthonormal vertical frame $(\zeta_1, \ldots, \zeta_k)$ that trivializes the vertical distribution $\mathcal{V}_1$ of $(G,g_1)$. Then $(\eps^{-1} \zeta_1, \ldots, \eps^{-1}\zeta_k)$ is an orthonormal vertical frame for the vertical distribution $\mathcal{V}_{\eps}$ of $G_{\eps}$. For any two horizontal vectors $X, Y$ we calculate
\begin{align*}
 A_{\eps}(X,Y) &= \frac{1}{2} \sum_{a = 1}^k g_{\eps}([X,Y], \eps^{-1} \zeta_a) \\
&= \frac{1}{2} \sum_{a  = 1}^k \eps^{2-1} \check{g}([X,Y], \zeta_a ) \\
& = \eps A_1(X,Y).
\end{align*}
In particular, it follows that $ \lim_{\eps \rightarrow 0}\Vert A_{\eps} \Vert_{g_{\eps}} = 0$. Hence, all assumptions of Corollary \ref{CorollaryExplicitDirac} are fulfilled. Thus, if for almost all $\eps \in (0,1]$ the space of affine parallel spinors is nontrivial then there is an induced spin structure on $B$ and the spectra of the Dirac operators restricted to the space of affine parallel spinors converges, up to multiplicity, to the spectrum of the Dirac operator $D^B$ of $B$, if $k$ or $\dim(B)$ is even, respectively to the spectrum of $D^B \oplus -D^B$, if $\dim(B)$ and $k$ are odd. 
\end{ex}

We conclude this section by comparing Theorem \ref{MainTheorem} with the results by Lott, Theorem \ref{LottSingular} and Theorem \ref{LottSmooth}.  The main differences between the strategies used in this article and \cite{LottDirac} is first the isometry $Q$ introduced in Lemma \ref{QIsometry} and that we did not use either the Bochner-type formula for the Dirac operator or the minimax characterization.  Nonetheless our results are compatible with the results of \cite{LottDirac}. Let $(M_i,g_i)_{i \in \bbN}$ be a collapsing sequence of spin manifolds in $\mfdspacenk$ converging to an $n$-dimensional Riemannian manifold $(B,h)$. First, we observe that the twisted Clifford bundle $\SpinB\otimes\mathcal{P}$ is the same as the Clifford module $E^B$ in Theorem \ref{LottSmooth}. Moreover, if we would have used the map $Q: L^2(\SpinB \otimes \mathcal{P}) \rightarrow \mathcal{S}^{\aff}$ without the factor $\sqrt{\vol(Z)}$ then we would have that any $\Phi \in \mathcal{S}^{\aff}$ satisfies
\begin{equation*}
\begin{aligned}
D^M \Phi &=  Q \circ \check{D}^{\mathcal{T}} \circ Q^{-1} \Phi + \frac{1}{2}\gamma(\mathcal{Z}) +  \frac{1}{2}\gamma(\mathcal{A})- \frac{1}{2}\sum_{a=1}^k \gamma(T(\zeta_a, \zeta_a) )\Phi 
\end{aligned}
\end{equation*}
This operator restricted to the space of affine spinors  is equivalent to the operator $\mathcal{D}^B$ of Theorem \ref{LottSmooth} in the case $G = \Spin(n+k)$ and $V = \Sigma_{n+k}$ \cite[(3.9)]{LottDirac}. 

Our results are also compatible with Theorem \ref{LottSingular} in the case of $G= \Spin(n+k)$, $V = \Sigma_{n+k}$ and a smooth limit space. We observe that the space $(L^2(\check{X}, \chi \dvol) \otimes V)^G$ is up to an isometry equivalent to $L^2(\SpinB \otimes \mathcal{P}, \dvol)$. Due to the isometry $Q$, see Lemma \ref{QIsometry}, we were able to choose the limit measure $\chi \equiv 1$.  The operator $ (D^X)^2$ arises as the limit of the operators $(D^{X_i})_{i \in \bbN}$, where $D^{X_i}$ is the Dirac operator on $M_i$ restricted to the space of affine parallel spinors \cite[(4.6) and p.\ 192]{LottDirac}. Thus, it follows that $(D^X)^2$ is, up to isometry, equivalent to the operator $(\mathcal{D}^B)^2$ derived in Theorem \ref{MainTheorem}. Hence, Theorem \ref{MainTheorem} is an accentuation of Lott's results \cite{LottDirac} in the special case of collapsing spin manifolds in $\mfdspacenk$ with a smooth limit space.

\bibliography{literature.bib}
\bibliographystyle{amsalpha}

\end{document}